\documentclass[11pt]{amsart}
\numberwithin{equation}{section}
\usepackage{latexsym}
\usepackage{amsmath}
\usepackage{amssymb}
\usepackage{mathrsfs}
\usepackage{graphicx,colordvi}
\usepackage{upgreek}
\usepackage{ifthen}
\usepackage{color}

\usepackage[T1]{fontenc}
\usepackage[latin1]{inputenc}

\numberwithin{equation}{section}

\newtheorem{defi}{Definition}[section]
\newtheorem{theorem}[defi]{Theorem}
\newtheorem{lemma}[defi]{Lemma}
\newtheorem{corollary}[defi]{Corollary}
\newtheorem{proposition}[defi]{Proposition}
\newtheorem{remark}[defi]{Remark}
\newtheorem{remarks}[defi]{Remarks}

\usepackage{latexsym}
\usepackage{amsmath}
\usepackage{amssymb}
\newcommand{\cA}{{\mathcal A}}

\newcommand{\cD}{{\mathcal D}}

\newcommand{\cB}{{\mathcal B}}

\newcommand{\cL}{{\mathcal L}}

\newcommand{\cT}{{\mathcal T}}

\newcommand{\CC}{{\mathbb C}}

\newcommand{\BB}{{\mathbb B}}
\newcommand{\EE}{{\mathbb E}}

\newcommand{\R}{{\mathbb R}}
\newcommand{\RR}{{\mathbb R}}

\newcommand{\sA}{{\mathsf{A}}}
\newcommand{\sB}{{\mathsf{B}}}
\newcommand{\sF}{{\mathsf{F}}}
\newcommand{\sH}{{\mathsf{H}}}
\renewcommand{\epsilon}{\varepsilon}

\DeclareMathOperator{\curl}{curl}
\DeclareMathOperator{\Tr}{Tr}
\DeclareMathOperator{\tr}{tr}
\DeclareMathOperator{\diver}{div}
\DeclareMathOperator{\real}{Re}
\DeclareMathOperator{\mdiv}{div}

\newcommand{\bb}{{\bb}}

\makeatletter
\@namedef{subjclassname@2020}{\textup{2020} Mathematics Subject Classification}
\makeatother

\begin{document}

\title[General Ericksen-Leslie Model]{Nematic liquid crystals: Ericksen-Leslie theory with general stress tensors}

\author{Matthias Hieber}
\address{TU Darmstadt\\
Fachbereich Mathematik\\
        Darmstadt, Germany}
\email{hieber@mathematik.tu-darmstadt.de}

\author{Jinkai Li}
\address{South China Research Center for Applied Mathematics and Interdisciplinary
Studies, School of Mathematical Sciences, South China Normal University, Guangzhou 510631, P. R. China}
\email{jklimath@m.scnu.edu.cn; jklimath@gmail.com}

\author{Mathias Wilke}
\address{Martin-Luther-Universit\"at Halle-Witten\-berg\\
         Institut f\"ur Mathematik \\
         Halle (Saale), Germany}
\email{mathias.wilke@mathematik.uni-halle.de}



\begin{abstract}
The Ericksen-Leslie model for nematic liquid crystal flows in case of an isothermal and incompressible fluid with general Leslie stress and anisotropic elasticity, i.e. with general  Ericksen stress tensor, is shown for the first time to be strongly well-posed.
Of central importance is a fully nonlinear boundary condition for the director field, which, in this generality, is necessary to guarantee that the system fulfills physical principles.
The system is shown to be locally, strongly well-posed in the $L_p$-setting. More precisely, the existence and uniqueness of a local,  strong $L_p$-solution to the general system is proved and it is shown
that the director $d$ satisfies  $|d|_2\equiv 1$ provided this  holds for its  initial data $d_0$. In addition, the solution is shown to depend continuously on the data.

The results are proven without any structural assumptions on the Leslie coefficients and in particular without assuming Parodi's relation.
\end{abstract}

\maketitle

\section{Introduction}

In physics there are various ways of describing order parameters in liquid crystals: Doi-Onsager -, Landau-De Gennes - and Ericksen-Leslie theory. They lead to mathematical theories at various
levels. The Ericksen-Leslie model is a so-called vector model. Another type of model describing liquid crystal flows is the Q-tensor model, including the Landau-De Gennes theory.
In contrast to vector models, it uses a traceless $3 \times 3$-matrix $Q$ to describe the alignment of molecules, see e.g. \cite{BE94} or \cite{SV12}.

In this article we concentrate on the Ericksen-Leslie model with general Leslie and general Ericksen stress. In  their pioneering
works, Ericksen and Leslie \cite{Eri62,Eri66,Les66,Les68} developed during the 1960's the continuum theory of nematic
liquid crystals. This theory models nematic
liquid crystal flows from a hydrodynamical point of view  and reduces to the Oseen-Frank theory in the static case. It describes the evolution of the
complete system under the influence of the velocity $u$ of the fluid and the orientation configuration $d$ of rod-like liquid crystals, see also \cite{BZ11}.
The original derivation \cite{Eri62,Les68} is based on the conservation  laws for mass, linear and angular momentums as well as on certain very specific
constitutive relations, which nowadays are called  the {\em Leslie} and {\em Ericksen stress}. For a very thorough description and investigation  of the Ericksen-Leslie model we refer to the
monographs by Virga \cite{Vir94} and Sonnet and Virga \cite{SV12}.

Due to the complexity of these systems, certain simplified systems were investigated frequently in the past. In fact, the rigorous analysis of the Ericksen-Leslie system  began with the work of
Lin \cite{Lin91} and Lin and Liu \cite{LL95}, who introduced and studied the nowadays called simplified isothermal system, see also \cite{HNPS16}.  For  well-posedness criteria concerning various
simplifications and various assumptions on the Leslie as well as Ericksen coefficients, we refer to  \cite{CS01,FFRS12,LW16,Wan11,WZZ13,WZZ15b} as well to  the
survey articles \cite{HiPr18,WZZ21} and the references therein.

It was a long outstanding open problem to decide whether the Ericksen-Leslie system subject to general Leslie and general Ericksen stress is well-posed in the weak or strong sense. First results in
this direction are going back to Wu, Xu and Liu \cite{WXL13,WZZ13} and Lin and Wang \cite{LW16}, who proved well-posedness results for the Ericksen-Leslie system under the assumptions that
the Leslie coefficients are satisfying certain assumptions related to Parodi's relation and where the Oseen-Frank free energy $\psi$ is simplified to the energy  functional for harmonic maps, i.e.,
$I(d)=\int_\Omega \psi(d,\nabla d)=\int_\Omega |\nabla d|^2dx$ due to the isotropy assumptions in  the Oseen-Frank functional described in detail below.

In this article we give an affirmative answer to this problem for the general case with general Leslie stress and anisotropic elasticity, i.e. general Ericksen stress. For the description of the full system, see e.g. Section 4.7 of \cite{HiPr18} or Section \ref{sec:Model_and_Results} below. Denoting by $\varrho$ the density and by $\theta$ the temperature of the fluid,
the free energy $\psi$ is of the form $\psi=\psi(\varrho,\theta,d,\nabla d)$, where $d$ is the director field. In the isothermal and incompressible situation, $\psi$ is given by the classical Oseen-Frank free energy
\begin{equation*}
\psi(d,\nabla d) = k_1({\rm div}\, d)^2 +k_2|d\times(\nabla\times d)|_2^2 + k_3 |d\cdot(\nabla \times d)|^2 + (k_2+k_4)[ {\rm tr}(\nabla d)^2-({\rm div}\, d)^2],
\end{equation*}
where $k_i$ are the so-called {\em Frank coefficients}. Based on  physical principles, the four Frank coefficients $k_1,k_2,k_3,k_4$ are all different in general, however, the first three ones
should be  strictly positive.
The Frank coefficients  are often assumed to satisfy the Ericksen inequalities
$$
k_1 >0, k_2 >0, k_3 >0, k_2> |k_4|, 2k_1 > k_2 + k_4,
$$
which are known to be necessary and sufficient for the inequality $\psi(d,\nabla d) \geq c |\nabla d|^2$ for all $d$ and some constant $c>0$, see \cite{Bal17}.

The first three terms in $\psi$ defined above, describe splay, twist and bend of the director field. The fourth term $k_2+k_4$, the saddle-splay term, is a null Lagrangian meaning that
$(k_2+k_4)\int_\Omega {\rm tr}(\nabla d)^2-({\rm div}\, d)^2dx$ depends only on the values of $d$ on the boundary $\partial \Omega$. Thus if
$d_{|_{\partial\Omega}}$ is prescribed, such as for Dirichlet boundary conditions, the term can be neglected with respect to  energy considerations. However, this is \emph{not} the
case when $d_{|_{\partial\Omega}}$ is only partially prescribed, as in the situation of weak anchoring boundary conditions or for \emph{fully nonlinear} boundary conditions as in our case and described in detail below, see \eqref{fnbc}.
For the general free energy $\psi(d,\nabla d)$, the general Ericksen stress tensor becomes
$$
S_E=- \varrho \frac{\partial \psi}{\partial(\nabla d)}[\nabla d]^{\sf T}.
$$
We remark that in the \emph{isotropic} case, i.e. if $k_1=k_2=k_3=1$ and $k_4=0$, the Oseen-Frank free energy
$\psi(d,\nabla d)$ coincides with $|\nabla d|^2$, hence $\frac{\partial \psi}{\partial(\nabla d)}=2\nabla d$,
which simplifies the problem.

In \cite{HiPr17}, strong well-posedness for the Ericksen-Leslie system in the incompressible case with general Leslie but \emph{isotropic} elasticity stress, i.e. if $k_1=k_2=k_3=1$ and $k_4=0$ in the Oseen-Frank free energy $\psi$, was proved for the first time without assuming
any structural conditions on the Leslie coefficients, as e.g. Parodi's relation \cite{Par70}. It was possible to prove a result of this type, since the   approach given in  \cite{HiPr17}
is based on the theory of quasilinear evolution equations  and not on  energy estimates. It seems that the latter requires certain dissipation rates and therefore
structural conditions on the Leslie coefficients are  needed, when pursuing  an approach based on energy estimates as in \cite{WXL13}, \cite{WZZ13}. In \cite{HiPr17,HiPr18} it is only assumed
that the six Leslie coefficients are smooth functions and that the coefficient $\mu_s$ (corresponding to $\alpha_4$ in \cite{LW16}) associated with the usual Cauchy stress tensor
is strictly positive, thus guaranteeing that the resulting Laplacian has the correct sign.



In the special case of $\Omega = \R^3$, Hong, Li and Xin \cite{HLX14} and Ma, Gong and Li \cite{MGL14} obtained a well-posedness result for anisotropic elasticity, however, for completely
vanishing Leslie stress $S_L$ and without stretching terms and with the assumption that $k_2 + k_4=0$.

It is the aim of this article to investigate for the first time the Ericksen-Leslie system with general Leslie stress $S_L$ and \emph{anisotropic} elasticity, i.e. with general Ericksen stress $S_E$ (see \eqref{eq:const} \& \eqref{eq:elgenincom}), in bounded domains $\Omega \subset \R^3$ with boundary $\partial\Omega\in C^3$. We show that this system is strongly well-posed without any structural assumptions on the Leslie coefficients. For the Frank coefficients, we assume that
\begin{equation}\label{eq:cond_Frank}
k_1 >0, k_2 >0, k_3 >0 \mbox{ and for } 0 < \alpha \leq \min\{k_1,k_2,k_3\} \mbox{ let } k_4= \alpha - k_2.
\end{equation}
Moreover, we suppose that at least one of the following conditions
\begin{equation}\label{eq:cond_Frank2}
9k_3>k_1\quad\mbox{or}\quad 2|k_1-k_3|<\min\{k_2,k_3\}
\end{equation}
is satisfied.

In the two-dimensional case and if $d$ is the heat flow of harmonic maps, solutions with finite time singularities are constructed in \cite{CDY92}, which yields blow-up in the $d$-equation.
Huang, Lin, Liu and Wang \cite{HLLW15} were able to construct examples, in case of $\Omega$ being the unit ball in $\R^3$ and of initial data $u_0,d_0$ having sufficiently small energy and
$d_0$ fulfilling a topological condition in the case of Dirichlet boundary conditions  $d=(0,0,1)^{\sf T}$, for which one has finite time blow-up of $(u,d)$. For a related result in two space dimensions,
we refer to \cite{LLWW22}.

On the other hand, recalling the results given in \cite{HiPr17} and \cite{HiPr18}, we already noted  that the Ericksen-Leslie system with isotropic elasticity and general Leslie coefficients, but classical
Neumann boundary conditions for $d$ is strongly well-posed. It was also shown in \cite{HP15a} and \cite{HiPr18} that in this case the Ericksen-Leslie system subject to Neumann boundary conditions
is thermodynamically consistent, meaning that it fulfills the second law of thermodynamics. In order to find boundary conditions for $d$ in the case of general, anisotropic elasticity which respect the underlying physics,
we refer for example to \cite{HP15a}. There it was shown that in the case of general elasticity certain fully nonlinear and natural boundary conditions for $d$ are needed in order to ensure that the system is consistent with
physical principles.
More precisely, it is shown in \cite[Section 15.2]{HP15a}, that the entropy production of the Ericksen-Leslie system is nonnegative, i.e. the \emph{second law of thermodynamics} is satisfied, provided that the energy flux $\Phi_e$ of the Ericksen-Leslie system is modeled by
$$\Phi_e=q+\pi u-Su-\frac{\partial\psi}{\partial(\nabla d)^{\sf T}}\cD_t d.$$
Here $u$ means velocity, $\pi$ pressure, $S$ extra stress (cf. \eqref{eq:const}), $q$ denotes the heat flux, $\psi$ is the Oseen-Frank free energy and $\cD_t=\partial_t+u\cdot\nabla$ is the Lagrangian derivative.

At the boundary $\partial\Omega$, energy should be preserved, meaning that $(\Phi_e|\nu)=0$, where $\nu$ denotes the unit normal vector field on $\partial\Omega$. Under the assumptions $(q|\nu)=0$ and $u=0$ at $\partial\Omega$, this readily implies
\begin{equation}\label{eq:nonlinBC0}
\left(\frac{\partial\psi}{\partial(\nabla d)} \cdot \nu \Big|\cD_t d \right)=0\quad\text{on}\quad\partial\Omega.
\end{equation}
As the director $d$ has length 1, it holds that $P_d\cD_t d=\cD_t d$, where $P_d=I-d\otimes d$. Therefore, \eqref{eq:nonlinBC0} is clearly valid, provided $d$ satisfies the natural boundary condition
\begin{equation}\label{fnbc}
 P_d\frac{\partial\psi}{\partial(\nabla d)} \cdot \nu= 0\quad \mbox{ on }\partial\Omega.
\end{equation}
For this reason we employ \eqref{fnbc} throughout this paper. Let us emphasize that this type of boundary condition has already been investigated in detail in the book of E. Virga, see (3.116) in \cite{Vir94}, using variational techniques and the Euler-Lagrange formalism. In \cite{Vir94}, this type of boundary condition is called \emph{no anchoring condition} for $d$. Moreover, citing \cite[page 132]{Vir94}, "it should be noted that boundary conditions are often responsible for the appearance of defects in liquid crystals". It is very interesting to see, that the \emph{no anchoring boundary condition} obtained by variational principles in \cite{Vir94}, coincides with the boundary condition \eqref{fnbc}, obtained by using the entropy principle and the principle of thermodynamical consistency.

Let us also note that we are investigating here the boundary condition \eqref{fnbc} for the first time with respect to well-posedness of the Ericksen-Leslie system.

Observe that, in general, \eqref{fnbc} is a  {\em fully nonlinear} boundary condition for $d$, see Section \ref{sec:BC_Er_Op} for details. On the contrary, in the isotropic case $k_1=k_2=k_3=1$, $k_4=0$, it holds that $\frac{\partial \psi}{\partial(\nabla d)}=2\nabla d$, hence in this case, pure Neumann boundary conditions for $d$ are natural. For a detailed discussion of other possible boundary conditions for $d$, we refer the reader e.g. to \cite[Section 3.5]{Vir94}.

From a mathematical point of view it is very satisfactory to see  that the natural (nonlinear) boundary condition \eqref{fnbc}, motivated originally  by  physical principles, yields the existence and uniqueness of strong solutions to the Ericksen-Leslie system subject to general Leslie stress and anisotropic elasticity, i.e., with general Ericksen stress, see Theorem \ref{thm:LWP2} for details.

There are several  major  difficulties arising in the investigation of the general Ericksen-Leslie system.
In a first key step, we will show that the associated Ericksen operator is normally elliptic in $\R^3$ under the assumption that the Frank
coefficients $k_1,\ldots,k_4$ satisfy the above condition \eqref{eq:cond_Frank}-\eqref{eq:cond_Frank2} but no other assumptions. Secondly, considering the situation of bounded domains, by physical principles, we are naturally  lead to the
above fully nonlinear boundary condition \eqref{fnbc} for $d$, which, however,  is analytically delicate. We master these difficulties by showing first that the linearized system satisfies
the Lopatinskii-Shapiro condition and thus that the general Ericksen-Leslie system constitutes a normally elliptic boundary value problem in the sense of \cite[Section 6]{PrSi16}.
On the technical side, we note that our proofs are using a broad range of  methods ranging from  operator-valued Fourier multipliers to Schur complements.


Our approach, based on modern quasilinear theory, yields strong local-in-time well-posedness of the general Ericksen-Leslie system.
Of course, $d$ satisfies $|d(t,x)|_2=1$ for all $t \in [0,T]$, for some $T>0$ and all
$x \in \Omega$, provided $\Omega$ is bounded. The approach also yields that the solution depends continuously on the data.

This article is organized as follows. In Section 2 we start with  a precise description of the Ericksen-Leslie model including the involved stress tensors as well as its natural boundary conditions
motivated by the second law of thermodynamics and we close Section 2 by stating the two main results of this article. Section 3 is devoted to the computation of the
Ericksen operator, based on the general Oseen-Frank free energy functional, while in Section 4, we introduce the functional analytic setting for our approach. In Section 5 we show that
the principal part of the linearized Ericksen operator subject to the principal part of the linearization of the boundary operator is strongly elliptic and satisfies the Lopatinskii-Shapiro condition, thus allowing an approach based on modern quasilinear theory. The results obtained in Section 5 are crucial for proving maximal regularity results in Section 6. Finally, in Section 7, we prove our two main results.

To formulate our main result in the next section, we define time weighted spaces $L_{p,\mu}(J;E)$ for a UMD-space $E$, $J = (0,T)$, $0 < T \le \infty$, $p \in (1,\infty)$, $\mu \in (\frac{1}{p},1]$ as
$$L_{p,\mu}(J;E) := \{u \colon J \to E \mid [t\mapsto t^{1-\mu} u(t)] \in L_p(J;E)\}$$
equipped with their natural  norms
$\|u \|_{L_{p,\mu}(J;E)} := \| [t\mapsto t^{1-\mu} u(t)] \big \|_{L_p(J;E)}$.
For $k \in \mathbb{N}_0$, the associated weighted Sobolev spaces are defined by
$$W_{p,\mu}^{k}(J;E) = H_{p,\mu}^{k}(J;E) := \{u \in W_{1,\mathrm{loc}}^{k}(J;E) \mid u^{(j)} \in L_{p,\mu}(J;E), \enspace j
\in \{0,\dots,k\}\}$$
and these spaces are equipped with their natural  norms.
For $s\in (0,1)$, the weighted Sobolev-Slobodeckii spaces $W_{p,\mu}^{s}(J;E)$ are defined as
$$W_{p,\mu}^{s}(J;E)=\{u\in L_{p,\mu}(J;E):\|u\|_{W_{p,\mu}^{s}(J;E)}<\infty\},$$
where
$$\|u\|_{W_{p,\mu}^{s}(J;E)}:=\|u\|_{L_{p,\mu}(J;E)}+[u]_{W_{p,\mu}^{s}(J;E)},$$
and
$$[u]_{W_{p,\mu}^{s}(J;E)}:=\left(\int_0^T\int_0^t\tau^{p(1-\mu)}\frac{\|u(t)-u(\tau)\|_E^p}{(t-\tau)^{sp+1}}d\tau dt\right)^{1/p},$$
see \cite[Formula (1.5)]{DuShaSi24}.
Furthermore, we
set $W_{p,\sigma}^{s}(\Omega;\RR^3):=W_{p}^{s}(\Omega;\RR^3)\cap L_{p,\sigma}(\Omega;\RR^3)$ for any $s>0$, where $L_{p,\sigma}(\Omega;\RR^3)$ denotes the space of solenoidal $L_p$-functions on $\Omega$ and $W_{p}^{s}(\Omega;\RR^3)$ is a classical Sobolev-Slobodeckii space, see e.g. \cite{Tri78,Tri83}.

\section{The general Ericksen-Leslie model and Main Results}\label{sec:Model_and_Results}

In their pioneering articles, Ericksen \cite{Eri62} and Leslie \cite{Les68}  developed a  continuum theory for the flow of  nematic liquid crystals based on the conservation  laws for
mass, linear and angular momentums as well as on certain constitutive relations.

The general incompressible Ericksen-Leslie model in the isothermal case reads as
\begin{align}\label{eq:el}
\left\{
\begin{array}{rlll}
 \partial_t u + (u \cdot \nabla) u  &\!=\!& \mdiv \sigma & \text{in } (0,T) \times \Omega,  \\
 \mdiv u &\!=\!& 0 & \text{in } (0,T) \times \Omega, \\
 d \times \left(g + \mdiv\left(\frac{\partial \psi }{\partial (\nabla d)}\right) - \nabla_d \psi  \right)&\!=\!& 0 & \text{in } (0,T) \times \Omega,\\
 |d|_2 &\!=\!& 1 & \text{in } (0,T) \times \Omega, \\
 (u,d)_{\vert t=0} &\!=\!& (u_0,d_0) & \text{in }   \Omega.
 \end{array}\right.
\end{align}
Here, $u$ denotes the velocity of the fluid, $d$ the so-called director, and $\sigma$ the stress tensor, given by
$$
\sigma := -   \pi I - \frac{\partial\psi}{\partial(\nabla d)} \nabla d + \sigma_L,
$$
where $\pi$ is the fluid pressure and
\begin{equation}\label{def:sigmal}
\sigma_L := \alpha_1 (d^TDd)d\otimes d +\alpha_2N \otimes d + \alpha_3 d\otimes N + \alpha_4 D + \alpha_5(Dd)\otimes d + \alpha_6 d \otimes (Dd)
\end{equation}
denotes the Leslie stress. Here  $\alpha_i$, $i=1,\ldots,6$ denote the Leslie viscosities,  $D = \frac{1}{2}([\nabla u]^{\sf T}+ \nabla u)$ the deformation tensor,
$$
N:=N(u,d):=\partial_td + (u \cdot \nabla)d - Vd,
$$
with $V=\frac{1}{2}(\nabla u -[\nabla u]^{\sf T})$ the vorticity tensor and $(a\otimes b)_{ij}:=a_ib_j$ for $1\leq i,j \leq 3$.
The kinematic transport of $d$ is denoted by $g$ and is given by
\begin{align}\label{def:g}
g := g(u,d) := \lambda_1 N + \lambda_2 Dd,
\end{align}
where $\lambda_1,\lambda_2 \in \R$. Moreover, the free energy $\psi$ is given by the classical Oseen-Frank free energy defined by
\begin{equation}\label{def:psi}
\psi(d,\nabla d) = k_1({\rm div}\, d)^2 +k_2|d\times(\nabla\times d)|_2^2 + k_3 |d\cdot(\nabla \times d)|^2 + (k_2+k_4)[ {\rm tr}(\nabla d)^2-({\rm div}\, d)^2],
\end{equation}
where $k_i$ are the so-called {\em Frank coefficients}.
The system has to be completed by suitable initial and boundary conditions.

Following arguments from thermodynamics and employing the entropy principle, the above Ericksen-Leslie model \eqref{eq:el} was extended in \cite{HP15a}  to the non-isothermal situation and to the case of
compressible fluids  in a thermodynamically consistent way. Let us emphasize that these extended  models {\em contain the classical Ericksen-Leslie model in its general form as a special case}.

Given a bounded domain $\Omega \subset \R^n$, $n \geq 2$, with smooth boundary,  the general Ericksen-Leslie model in the  non-isothermal situation derived as in
\cite{HP15a, HiPr18},  reads as
\begin{align}\label{eq:elgeneral}
\left\{
\begin{array}{rll}
\partial_t \rho +{\rm div}(\rho u)&=0\quad &\mbox{in } \Omega,\\
\rho\cD_tu +\nabla \pi &= {\rm div}\, S \quad &\mbox{in } \Omega,\\
\rho\cD_t\epsilon +{\rm div}\, q &= S:\nabla u -\pi{\rm div}\, u+{\rm div}(\rho\partial_{\nabla d}\psi\cD_td)\quad &\mbox{in } \Omega,\\
  \gamma \cD_td-\mu_V Vd&= P_d\big( {\rm div}(\rho\frac{\partial\psi}{\partial(\nabla d)})-\rho\nabla_d\psi\big)+ \mu_D P_d Dd \quad &\mbox{in } \Omega,\\
  u=0,\quad q\cdot \nu&=0 \quad  &\mbox{on } \partial\Omega,\\
\rho(0)=\rho_0,\quad u(0)=u_0,\quad \theta(0)&=\theta_0,\quad d(0)=d_0\quad &\mbox{in } \Omega.
\end{array}\right.
\end{align}
Here the unknown variables $\rho, u,\pi$ denote the density, velocity and pressure of the fluid, respectively, $\epsilon$ the
internal energy and $d$ the so called director, which must have modulus 1.  Moreover, $q$ denotes the heat flux,  $\cD_t=\partial_t +u\cdot\nabla$ the Lagrangian derivative and
$P_d$ is defined as $P_d=I-d\otimes d$.
These equations have to be supplemented by the thermodynamical laws
\begin{align}\label{eq:thermo}
\epsilon = \psi +\theta \eta, \quad \eta = -\partial_\theta \psi, \quad \kappa = \partial_\theta \epsilon,\quad \pi =\rho^2 \partial_\rho\psi,
\end{align}
and by the constitutive laws
\begin{align}\label{eq:const}
\left\{
\begin{array}{lll}
S &= S_N + S_E + S_L, \\ 
S_N&= 2\mu_s D + \mu_b {\rm div}\, u \, I, \\
S_E&=  -\rho \frac{\partial \psi}{\partial\nabla d}[\nabla d]^{\sf T},\\
S_L&= S_L^{stretch} + S_L^{diss}, \\
 S_L^{stretch}& = \frac{\mu_D+\mu_V}{2\gamma}{\sf n} \otimes d + \frac{\mu_D-\mu_V}{2\gamma} d \otimes {\sf n},\quad
 {\sf n}= \mu_V Vd +\mu_D P_dDd-\gamma\cD_td,\\
S_L^{diss}& =\frac{\mu_P}{\gamma} ({\sf n}\otimes d + d\otimes {\sf n}) + \frac{\gamma\mu_L+\mu_P^2}{2\gamma}(P_dDd\otimes d + d\otimes P_dDd) +\mu_0 (Dd|d)d\otimes d,\\
\end{array}\right
.\end{align}
and
\begin{equation}\label{eq:q1}
q = -\tilde\alpha_0 \nabla\theta -\tilde\alpha_1(d|\nabla\theta)d.
\end{equation}
Here all coefficients $\mu_j,\tilde\alpha_j$ and $\gamma$ are functions of $\rho,\theta,d,\nabla d$.
For thermodynamical consistency the following conditions are required
\begin{align}\label{eq:pos1}
\mu_s\geq0,\quad 2\mu_s+n\mu_b\geq0,\quad \tilde\alpha_0\geq0,\quad \tilde\alpha_0+\tilde\alpha_1\geq0,\quad \mu_0,\mu_L\geq0,\quad \gamma>0.
\end{align}
We also recall from \eqref{fnbc} that the natural boundary condition at $\partial\Omega$ for $d$ becomes
\begin{equation}\label{eq:nonlinbc}
 P_d\frac{\partial \psi}{\partial(\nabla d)} \cdot \nu = 0\quad \mbox{ on }\partial\Omega,
\end{equation}
which, in general, is {\em fully nonlinear}.

In the case of {\em isotropic elasticity} with constant density and constant temperature one has $k_1=k_2=k_3=1$ and $k_4=0$ and so the Oseen-Frank energy reduces to the Dirichlet energy, i.e.
$$
\psi(d,\nabla d)= |\nabla d|^2,
$$
and  thus
$\mdiv ( \frac{\partial \psi}{\partial( \nabla d)}) = 2\Delta d$.
Then the Ericksen stress tensor simplifies to
\begin{equation}\label{eq:ericksenstress}
S_E=  -\lambda \nabla d[\nabla d]^{\sf T},
\end{equation}
where $\lambda = \rho\partial_\tau\psi$, $\tau=\frac{1}{2}|\nabla d|^2$,
and the natural boundary condition at $\partial\Omega$ for $d$ becomes the Neumann boundary condition $\partial_\nu d = 0$ on $\partial\Omega$.

We emphasize that in the case $\mu_V=\gamma$, our parameters $\mu_s,\mu_0,\mu_V,\mu_D,\mu_P,\mu_L$ are in {\em one-to-one correspondence} to the celebrated
Leslie parameters $\alpha_1,\ldots,\alpha_6$ given in the Leslie stress $\sigma_L$ defined  as in \eqref{def:sigmal},
where $D,N$ and $V$ are defined as above.  This shows that our model \eqref{eq:elgeneral}-\eqref{eq:const} contains the classical isothermic and incompressible
Ericksen-Leslie model as a particular case.

It is the aim of this article to investigate the Ericksen-Leslie system with general Leslie stress and anisotropic elasticity, i.e. with general Ericksen stress tensor, analytically, in the isothermal situation and
for incompressible fluids.

The system then  reads as
\begin{align}\label{eq:elgenincom}
\left\{
\begin{array}{rll}
\rho\cD_tu +\nabla \pi &= {\rm div}\, S \quad &\mbox{in } \Omega,\\
\mdiv u  &= 0 \quad &\mbox{in } \Omega,\\
  \gamma \cD_td-\mu_V Vd&= P_d\big( {\rm div}(\rho\frac{\partial\psi}{\partial(\nabla d)})-\rho\nabla_d\psi\big)+ \mu_D P_d Dd \quad &\mbox{in } \Omega,\\
P_d \frac{\partial \psi}{\partial(\nabla d)} \cdot \nu &=0 & \mbox{on } \partial\Omega,\\
u&=0  &\mbox{on } \partial\Omega,\\
\end{array}\right.
\end{align}
subject to initial data $(u(0),d(0)) =(u_0,d_0)$. Here $\cD_t$, $P_d$ are defined as above and $S$, $\psi$ are defined as in \eqref{eq:const} \& \eqref{def:psi}.
We show in Section \ref{sec:BC_Er_Op} that the  fully nonlinear boundary condition for $d$ reads as
\begin{align*}
P_d\frac{\partial{\psi}}{\partial(\nabla d)} \cdot \nu &=2k_3\nabla d\cdot \nu+2P_d(k_1\diver d\cdot I-k_3(\nabla d)^{\sf T})\cdot \nu
+2(k_2-k_3)(d\cdot \curl d)(d\times\nu)\\
&\quad +2(k_2+k_4)P_d((\nabla d)^{\sf T}-\diver d\cdot I)\cdot \nu.
\end{align*}
We note furthermore, that all coefficients {$\mu_j,\gamma$} are functions of {$d,\nabla d$}.

Let us state our assumptions on the coefficients in the Ericksen-Leslie model with general Leslie and Ericksen stress tensor.

\begin{itemize}
  \item[(R)] The Leslie coefficients $\mu_j$ for $j\in\{s,V,D,P,L,0\}$ and the parameter $\gamma$ may depend on $d,\nabla d$ and are assumed to be smooth,
  \item[(P)] It holds that $\mu_s>0$, $\gamma>0$ and $\mu_j\ge 0$ for $j\in\{0,L\}$,
  \item[(F)] The Frank coefficients satisfy $k_j>0$ for $j\in\{1,2,3\}$ and for $0<\alpha\le\min\{k_1,k_2,k_3\}$ let $k_4=\alpha-k_2$. Furthermore, let at least one of the conditions
    \begin{itemize}
      \item[$\bullet$] $9k_3>k_1$,
      \item[$\bullet$] $2|k_1-k_3|<\min\{k_2,k_3\}$,
    \end{itemize}
     be satisfied.
\item[(B)] The initial data $d_0$ satisfies the compatibility condition $\cB_{d_0}(\nabla)d_0 =0$ on $\partial\Omega$,  where
\begin{align*}
\frac{1}{2}\cB_{d_0}(\nabla)d_0 &:= k_3\nabla d_0\cdot \nu+P_{{d}_0}(k_1\diver d_0\cdot I-k_3(\nabla d_0)^{\sf T})\cdot \nu\\
& \quad +(k_2-k_3)({d}_0\cdot \curl d_0)({d}_0\times\nu)
+(k_2+k_4)P_{{d}_0}((\nabla d_0)^{\sf T}-\diver d_0\cdot I)\cdot \nu.
\end{align*}
 \end{itemize}
In particular, note that we do {\em not} assume Parodi's relations. We note furthermore, that assuming  condition (F), the Oseen-Frank functional $\psi$ is positive, i.e. we have
$\psi(d,\nabla d)\ge 0$ for all $d\in \RR^3$ with $|d|_2=1$.

\begin{theorem}\label{thm:LWP2}
Let $1<p<\infty$, $\mu \in (\frac12 + \frac{5}{2p},1]$, $\Omega\subset\R^3$ be a bounded domain with boundary $\partial\Omega\in C^3$, and assume that  the  assumptions (R), (P) and (F) are  satisfied.
Then, given
$$(u_0,d_0)\in W_{p,\sigma}^{2\mu-2/p}(\Omega;\RR^3) \times W_{p}^{2\mu+1-2/p}(\Omega;\RR^3),$$
satisfying  $|d_0(x)|_2= 1$ for all $x\in\Omega$, the compatibility conditions (B) as well as  $u_0=0$ on $\partial\Omega$,
there exists $T=T(u_0,d_0)>0$ such that problem \eqref{eq:elgenincom} has a unique, strong solution
\begin{align*}
  u&\in \EE_{1,\mu}^u(0,T):= H_{p,\mu}^1((0,T);L_p(\Omega;\RR^3))\cap L_{p,\mu}((0,T);H_p^2(\Omega;\RR^3)), \\
 d&\in \EE_{1,\mu}^d(0,T):=H_{p,\mu}^1((0,T);H_p^1(\Omega;\RR^3))\cap L_{p,\mu}((0,T);H_p^3(\Omega;\RR^3))\\
    \nabla\pi&\in L_{p,\mu}((0,T);L_p(\Omega;\RR^3)).
  \end{align*}
Moreover, $d$ satisfies $|d(t,x)|_2=1$ for all $(t,x)\in [0,T]\times\Omega$ and the solution can be extended to a maximal solution with a maximal  interval of existence $[0,T_+(u_0,d_0))$.
\end{theorem}

\begin{remarks}\label{rem:ThmLWP2}\mbox{}{\rm

\begin{itemize}
\item[a)]
Note that in the case of isotropic elasticity, i.e., if $k_1=k_2=k_3$ and $k_4=0$, the nonlinear boundary condition \eqref{eq:nonlinbc} reduces to the classical linear Neumann boundary condition $\partial_\nu d= 0$ on
  $\partial\Omega$.
\item [b)]
We emphasize that the boundary condition \eqref{eq:nonlinbc} is not a mathematical construction but occurs as a necessary condition for proving that the non-isothermal and compressible Ericksen-Leslie
system is thermodynamically consistent meaning  that the associated entropy production rate is positive and thus satisfies the second law of thermodynamics only if the boundary condition
\eqref{eq:nonlinbc} is prescribed. For details we refer to \cite{HP15a} and \cite{HiPr18}.
\end{itemize}
}
\end{remarks}

We also prove that the solution of \eqref{eq:elgenincom}  depends continuously on the initial data $(u_0,d_0)$. To formulate  the result precisely, let us define the sets
$$
X_{\gamma,\mu}:= \{ u \in W_{p,\sigma}^{2\mu-2/p}(\Omega;\RR^3): u=0 \mbox{ on } \partial\Omega\} \times W_{p}^{2\mu+1-2/p}(\Omega;\RR^3) \mbox{ and }
$$
and
$$
\mathcal{M}:=\{(u,d)\in X_{\gamma,\mu}: |d|_2= 1 \mbox{ in } \Omega \mbox{ and } \cB_{d}(\nabla)d=0 \mbox{ on } \partial\Omega \},
$$
where $\cB_d(\nabla)d$ is defined as in (B) but with $d_0$ being replaced by $d$.
We also set $\EE_{1,\mu}(0,T) = \EE_{1,\mu}^u(0,T) \times \EE_{1,\mu}^d(0,T)$.
\goodbreak
\begin{theorem}\label{thm:contdep}
Let $1<p<\infty$, $\mu \in (\frac12 + \frac{5}{2p},1]$, $\Omega\subset\R^3$ be a bounded domain with boundary $\partial\Omega\in C^3$, and assume the  assumptions (R), (P) and (F).  For $(u_0,d_0)\in\mathcal{M}$ consider the unique solution of \eqref{eq:elgenincom}
with   maximal time of existence  $T_+(u_0,d_0)>0$ given by Theorem \ref{thm:LWP2}.

Then, for any $T\in (0,T_+(u_0,d_0))$, there exists $r>0$ such that for each $(\tilde{u}_0,\tilde{d}_0)\in \mathcal{M}\cap \BB_{X_{\gamma,\mu}}((u_0,d_0),r)$, the unique solution
$
(\tilde{u},\tilde{d},\tilde\pi)=\left(\tilde{u}(\cdot,(\tilde{u}_0,\tilde{d}_0)),\tilde{d}(\cdot,(\tilde{u}_0,\tilde{d}_0)),\tilde{\pi}(\cdot,(\tilde{u}_0,\tilde{d}_0))\right)
$
of \eqref{eq:elgenincom} with initial value $(\tilde{u}_0,\tilde{d}_0)$ satisfies
$$
(\tilde{u},\tilde{d},\tilde\pi)\in \mathbb{E}_{1,\mu}(0,T)\times L_{p,\mu}((0,T);\dot{H}_p^1(\Omega))
$$
and the mapping
\begin{align*}
  \mathcal{M}\cap \BB_{X_{\gamma,\mu}}((u_0,d_0),r) &\to \mathbb{E}_{1,\mu}(0,T)\times L_{p,\mu}((0,T);\dot{H}_p^1(\Omega)), \\
  (\tilde{u}_0,\tilde{d}_0)&\mapsto \left(\tilde{u}(\cdot,(\tilde{u}_0,\tilde{d}_0)),\tilde{d}(\cdot,(\tilde{u}_0,\tilde{d}_0)),\tilde{\pi}(\cdot,(\tilde{u}_0,\tilde{d}_0))\right)
\end{align*}
is continuous.
\end{theorem}

\section{The Ericksen operator}
In this section we compute the nonlinear operator associated to  ${\rm div}(\rho\frac{\partial\psi}{\partial(\nabla d)})$ as well as its boundary condition given in \eqref{eq:elgenincom}.

\subsection{Computation of the Ericksen-operator}\label{sec:compEop}\mbox{}

In the following, We compute the nonlinear operator associated to ${\rm div}(\rho\frac{\partial\psi}{\partial(\nabla d)})$ defined in \eqref{eq:elgenincom}, with $\psi$ from \eqref{def:psi}.
For a smooth vector field $d$, its derivative $\nabla d\in \RR^{3\times 3}$ and rotation $\curl d\in\RR^3$ are defined by
$\nabla d=(\partial_j d_i)_{i,j}$ and
$$
\curl d=(\partial_2 d_3-\partial_3 d_2, \partial_3 d_1-\partial_1 d_3, \partial_1 d_2-\partial_2 d_1)^{\sf T}.$$
It follows that
$$\Tr (\nabla d)^2=\sum_{i,k=1}^3\partial_k d_i\,\partial_i d_k,$$
where $\Tr A=\sum_{i=1}^3a_{ii}$ is the trace of a Matrix $A=(a_{ij})\in\RR^{3\times 3}$. This readily implies
$$\frac{\partial \Tr(\nabla d)^2}{\partial(\nabla d)}=2(\nabla d)^{\sf T}.$$
Furthermore, we have
$$\frac{\partial(\diver d)^2}{\partial(\nabla d)}=2\diver d\cdot \frac{\partial\diver d}{\partial(\nabla d)}=2\diver d\cdot I_{3\times 3}.$$
For a differentiable matrix $A=(a_{ij})\in\RR^{3\times 3}$, its divergence $\diver A\in\mathbb{R}^3$ is defined by
$\diver A=\left(\sum_{j=1}^3\partial_j a_{ij}\right)_{i=1,2,3}$.
We note that this readily implies
\begin{align*}
  \diver\left(\frac{\partial \Tr(\nabla d)^2}{\partial(\nabla d)}-\frac{\partial(\diver d)^2}{\partial(\nabla d)}\right)&=2\diver (\nabla d)^{\sf T}-2\diver\left(\diver d\cdot I_{3\times 3}\right)
  =2\nabla \diver d-2\nabla \diver d=0,
\end{align*}
hence
$$\diver \left(\frac{\partial\psi}{\partial(\nabla d)}\right)=\diver\left(\frac{\partial\tilde{\psi}}{\partial(\nabla d)}\right),$$
where
$$\tilde{\psi}(d,\nabla d)=k_1(\diver d)^2+k_2(d\cdot\curl d)^2+k_3|d\times \curl d|^2.$$
For convenience, we will rewrite the expression for $\tilde{\psi}$ as follows:
\begin{align*}
\tilde{\psi}(d,\nabla d)&=k_1(\diver d)^2+k_2(d\cdot\curl d)^2+k_3|d\times \curl d|^2\\
&=k_1(\diver d)^2+k_3|\curl d|_2^2+(k_2-k_3)(d\cdot\curl d)^2.
\end{align*}
Here we have used the property $|d|_2=1$ and the fact
$|a\cdot b|^2+|a\times b|_2^2=|a|_2^2\cdot |b|_2^2$ for all $a,b\in\RR^3$.

Since
$ |\curl d|_2^2=(\partial_2 d_3-\partial_3 d_2)^2+(\partial_3 d_1-\partial_1 d_3)^2+(\partial_1 d_2-\partial_2 d_1)^2$,
we obtain
$$
\frac{\partial|\curl d|^2}{\partial(\nabla d)}=2(\nabla d-(\nabla d)^{\sf T}).
$$
This readily implies that
\begin{align*}
\diver\Big(\frac{\partial(k_1(\diver d)^2+k_3|\curl d|^2)}{\partial(\nabla d)}\Big)&=2k_1\diver(\diver d\cdot I_{3\times 3})+2k_3\diver (\nabla d-(\nabla d)^{\sf T})\\
&=2k_1\nabla\diver d+2k_3\Delta d-2k_3\nabla\diver d\\
&=2k_3\Delta d+2(k_1-k_3)\nabla\diver d.
\end{align*}
Next, we will take care about the term $(d\cdot \curl d)^2$. In a first step, we obtain
\begin{equation}\label{eq:div_d_curld}
\frac{\partial(d\cdot \curl d)}{\partial(\nabla d)}=
\begin{pmatrix}
0 & -d_3 & d_2\\
d_3 & 0 & -d_1\\
-d_2 & d_1 & 0
\end{pmatrix},
\end{equation}
by employing the definition of $\curl d$. This in turn implies
\begin{align*}
\diver\Big(\frac{\partial(d\cdot \curl d)^2}{\partial(\nabla d)}\Big)&=2\diver\Big((d\cdot \curl d)\frac{\partial(d\cdot \curl d)}{\partial(\nabla d)}\Big)\\
&=2(d\cdot \curl d)\diver\Big(\frac{\partial(d\cdot \curl d)}{\partial(\nabla d)}\Big)+2\frac{\partial(d\cdot \curl d)}{\partial(\nabla d)}\cdot (\nabla(d\cdot \curl d))^{\sf T}.
\end{align*}
Here we have used the fact that $\diver (f\cdot A)=f\cdot\diver A+A\cdot (\nabla f)^{\sf T}$
for any differentiable matrix $A$ and any differentiable scalar $f$ with derivative
$\nabla f=(\partial_1 f,\partial_2 f,\partial_3 f)$.
A direct computation yields
$$\diver\left(\frac{\partial(d\cdot \curl d)}{\partial(\nabla d)}\right)=-\curl d,$$
by \eqref{eq:div_d_curld}, and hence
\begin{align*}
\diver\left(\frac{\partial(d\cdot \curl d)^2}{\partial(\nabla d)}\right)&=2\diver\left((d\cdot \curl d)\frac{\partial(d\cdot \curl d)}{\partial(\nabla d)}\right)\\
&=-2(d\cdot \curl d)\curl d+2\begin{pmatrix}
0 & -d_3 & d_2\\
d_3 & 0 & -d_1\\
-d_2 & d_1 & 0
\end{pmatrix}\cdot (\nabla(d\cdot \curl d))^{\sf T}.
\end{align*}
The last term can be rewritten as
\begin{align*}
\begin{pmatrix}
0 & -d_3 & d_2\\
d_3 & 0 & -d_1\\
-d_2 & d_1 & 0
\end{pmatrix}\cdot (\nabla(d\cdot \curl d))^{\sf T}&=(d\times\nabla)(d\cdot \curl d)\\
&=\left((d\times\nabla)\otimes d\right)\cdot \curl d+\left((d\times\nabla)\otimes \curl d\right)\cdot  d,
\end{align*}
where $a\otimes b=(a_ib_j)_{i,j}\in\RR^{3\times 3}$ for $a,b\in\RR^3$ and $\nabla=(\partial_1,\partial_2,\partial_3)^{\sf T}$.

Altogether, this yields
\begin{align*}
\diver\left(\rho\frac{\partial\tilde{\psi}}{\partial(\nabla d)}\right)&=2k_3\Delta d+2(k_1-k_3)\nabla\diver d +2(k_2-k_3)\left((d\times\nabla)\otimes \curl d\right)\cdot  d\\
  &\quad +2(k_2-k_3)\left((d\times\nabla)\otimes d\right)\cdot \curl d
  -2(k_2-k_3)(d\cdot \curl d)\curl d
\end{align*}
It remains to apply the projection $P_d=I-d\otimes d$. Since $|d|_2=1$, we obtain  $\Delta d\cdot d=-|\nabla d|_2^2$ and  hence
$$P_d(\Delta d)=\Delta d+|\nabla d|_2^2d.$$
Furthermore, a direct computation yields
\begin{equation}\label{eq:perp_d}
\left((d\times\nabla)\otimes \curl d\right)\cdot  d\perp d \mbox{ and } \left((d\times\nabla)\otimes d\right)\cdot \curl d\perp d.
\end{equation}
Therefore
$$
P_d\left(\left((d\times\nabla)\otimes \curl d\right)\cdot  d\right)=\left((d\times\nabla)\otimes \curl d\right)\cdot  d
$$
as well as
$$
P_d\left(\left((d\times\nabla)\otimes d\right)\cdot \curl d\right)=\left((d\times\nabla)\otimes d\right)\cdot \curl d.
$$
Summarizing, we arrive at the expression
\begin{align*}
  P_d\diver\Big(\frac{\partial\tilde{\psi}}{\partial(\nabla d)}\Big)&=2k_3\Delta d+2(k_1-k_3)P_d\nabla\diver d
  +2(k_2-k_3)\left((d\times\nabla)\otimes \curl d\right)\cdot  d\\
  & \quad +2(k_2-k_3)\left((d\times\nabla)\otimes d\right)\cdot \curl d
  -2(k_2-k_3)(d\cdot \curl d)P_d\curl d+2k_3|\nabla d|_2^2d.
\end{align*}
Finally, we note that  the principle part of the linearized Ericksen operator is then given by
\begin{equation}\label{eq:linEL}
\mathcal{A}_{\tilde{d}}(\nabla)d=2k_3\Delta d+2(k_1-k_3)P_{\tilde{d}}\nabla\diver d+2(k_2-k_3)\left((\tilde{d}\times\nabla)\otimes \curl d\right)\cdot  \tilde{d},
\end{equation}
for some given $\tilde{d}$.

\subsection{Boundary condition for the Ericksen operator}\label{sec:BC_Er_Op}\mbox{}

As explained in \eqref{fnbc}, the natural boundary condition for the Ericksen operator reads as
\begin{equation}\label{eq:BC_D}
P_d\frac{\partial\psi}{\partial(\nabla d)}\cdot\nu=0,
\end{equation}
with $\psi$ from \eqref{def:psi}.
We first compute $\frac{\partial\psi}{\partial(\nabla d)}\cdot\nu$. The results in Section \ref{sec:compEop} yield
$$
\frac{\partial(\diver d)^2}{\partial(\nabla d)}\cdot\nu=2(\diver d) \nu, \quad \frac{\partial \Tr(\nabla d)^2}{\partial(\nabla d)}\cdot\nu=2(\nabla d)^{\sf T}\cdot\nu, \quad
\frac{\partial|\curl d|^2}{\partial(\nabla d)}\cdot\nu=2(\nabla d-(\nabla d)^{\sf T})\cdot\nu,
$$
and
$$\frac{\partial(d\cdot \curl d)^2}{\partial(\nabla d)}\cdot\nu=2(d\cdot \curl d)
\begin{pmatrix}
0 & -d_3 & d_2\\
d_3 & 0 & -d_1\\
-d_2 & d_1 & 0
\end{pmatrix}\cdot\nu=2(d\cdot \curl d)(d\times\nu).$$
Summarizing, this yields
\begin{align*}
\frac{\partial\psi}{\partial(\nabla d)}\cdot\nu&=2k_3\nabla d\cdot \nu+2(k_1\diver d\cdot I-k_3(\nabla d)^{\sf T})\cdot \nu
+2(k_2-k_3)(d\cdot \curl d)(d\times\nu)\\
& \quad +2(k_2+k_4)((\nabla d)^{\sf T}-\diver d\cdot I)\cdot \nu.
\end{align*}
We observe that
$P_d(\nabla d\cdot \nu)=\nabla d\cdot\nu$ since $d^{\sf T}\cdot\nabla d\cdot\nu=\nu^{\sf T}\cdot(\nabla d)^{\sf T}\cdot d=0$,
as $|d|_2=1$. Furthermore, we have $P_d(d\times\nu)=d\times\nu$.
This finally yields
\begin{align}\label{eq:BC1}
 \begin{split}
P_d\frac{\partial\psi}{\partial(\nabla d)}\cdot\nu&=2k_3\nabla d\cdot \nu+2P_d(k_1\diver d\cdot I-k_3(\nabla d)^{\sf T})\cdot \nu
+2(k_2-k_3)(d\cdot \curl d)(d\times\nu)\\
&\quad +2(k_2+k_4)P_d((\nabla d)^{\sf T}-\diver d\cdot I)\cdot \nu.
\end{split}
\end{align}
We note that the principle part of the linearization $\mathcal{B}_{\tilde{d}}(\nabla)d$ of the right hand side of \eqref{eq:BC1} reads
\begin{align}\label{eq:BC2}
\begin{split}
\mathcal{B}_{\tilde{d}}(\nabla)d&=2k_3\nabla d\cdot \nu+2P_{\tilde{d}}(k_1\diver d\cdot I-k_3(\nabla d)^{\sf T})\cdot \nu
+2(k_2-k_3)(\tilde{d}\cdot \curl d)(\tilde{d}\times\nu)\\
&\quad +2(k_2+k_4)P_{\tilde{d}}((\nabla d)^{\sf T}-\diver d\cdot I)\cdot \nu,
\end{split}
\end{align}
for some given $\tilde{d}$.

\section{Functional analytic setting}\label{sec:FAsett}

Let $T\in (0,\infty)$ and $\Omega\subset\RR^3$ be a bounded domain with boundary $\partial\Omega\in C^3$. For the velocity field $u$, we choose the base space
$$L_{p,\mu}((0,T);L_p(\Omega;\RR^3)),$$
where $1<p<\infty$ and $\mu\in (1/p,1]$. In the equation for $u$, the terms of highest order are $\partial_t u$ for the variable $t\in (0,T)$ and $\partial_{x_i}\partial_{x_j}u$ for the variable $x\in\Omega$. Therefore, the optimal regularity class for $u$ is given by
$$H_{p,\mu}^1((0,T);L_p(\Omega;\RR^3))\cap L_{p,\mu}((0,T);H_p^2(\Omega;\RR^3)),$$
whereas
$$\nabla\pi\in L_{p,\mu}((0,T);L_p(\Omega;\RR^3))$$
is optimal for $\pi$. Since $\mathcal{D}_t d$ is a part of  $S_L$ and since $\diver S_L$ appears in the equation for $u$, it is natural to assume
$$d\in H_{p,\mu}^1((0,T);H_p^1(\Omega;\RR^3)).$$
In Section \ref{sec:compEop}, we explicitly computed the Ericksen operator $P_d\diver\left(\frac{\partial{\psi}}{\partial(\nabla d)}\right)$,
which acts as a quasilinear differential operator of second order.
Hence, choosing
$$L_{p,\mu}((0,T);H_p^1(\Omega;\RR^3))$$
as the base space for the $d$-equation, it follows that
$$H_{p,\mu}^1((0,T);H_p^1(\Omega;\RR^3))\cap L_{p,\mu}((0,T);H_p^3(\Omega;\RR^3))$$
is the optimal regularity class for the function $d$.

We note that for $\ell\in\{0,1\}$, the embeddings
\begin{equation*}
H_{p,\mu}^1((0,T);H_p^\ell(\Omega;\RR^3))\cap L_{p,\mu}((0,T);H_p^{2+\ell}(\Omega;\RR^3)) \hookrightarrow C([0,T];B_{pp}^{2\mu+\ell-2/p}(\Omega;\RR^3))
\end{equation*}
imply that, necessarily,
$$u(0)\in B_{pp}^{2\mu-2/p}(\Omega;\RR^3)\quad\text{and}\quad d(0)\in B_{pp}^{2\mu+1-2/p}(\Omega;\RR^3).$$
Here $B_{pq}^{s}(\Omega;\RR^3)$, $1<p,q<\infty$, $s>0$, denotes a classical Besov space, see e.g. \cite{Tri78,Tri83}.

With a view on the nonlinear boundary term $P_d\frac{\partial\psi}{\partial(\nabla d)}\cdot\nu$ computed in Section \ref{sec:BC_Er_Op}, we will also need the optimal regularity class
for the trace of $\partial_j d$, $j\in\{1,2,3\}$ on $\partial\Omega$. Given
$$d\in H_{p,\mu}^1((0,T);H_p^1(\Omega;\RR^3))\cap L_{p,\mu}((0,T);H_p^3(\Omega;\RR^3)),$$
it follows that
$$\partial_j d\in H_{p,\mu}^1((0,T);L_p(\Omega;\RR^3))\cap L_{p,\mu}((0,T);H_p^2(\Omega;\RR^3))$$
and hence
$$\tr_{\partial\Omega}\partial_j d\in
W_{p,\mu}^{1-1/2p}((0,T);L_p(\partial\Omega;\RR^3))\cap L_{p,\mu}((0,T);W_{p}^{2-1/p}(\partial\Omega;\RR^3)),
$$
for each $j\in\{1,2,3\}$.
For the remainder of this article, we will always assume that
\begin{equation}\label{eq:Assum_for_p}
1<p<\infty\quad\text{and}\quad\frac{1}{2}+\frac{5}{2p}<\mu\le 1,
\end{equation}
so that
$$W_{p}^{2\mu+\ell-2/p}(\Omega;\RR^3)=B_{pp}^{2\mu+\ell-2/p}(\Omega;\RR^3),$$
as $2\mu+\ell-2/p\notin\mathbb{N}$ for $\ell\in\{0,1\}$.
Moreover, \eqref{eq:Assum_for_p} ensures the embedding
$$
W_{p}^{2\mu+\ell-2/p}(\Omega;\RR^3)\hookrightarrow C^{1+\ell}(\overline{\Omega};\RR^3),\quad \ell\in\{0,1\}.
$$
Setting $z:=(u,d)$, we rewrite \eqref{eq:elgenincom} together with its boundary and initial conditions in the equivalent form
\begin{equation}
\label{eq:iEL2}
\left\{\begin{aligned}
\rho\partial_tu+A_u(z)u +R_1^{\sf T}(z)\partial_t d+\nabla \pi &=F_u(z)&&\text{in}&&(0,T)\times\Omega ,\\
-R_0(z)u+\gamma(z)\partial_td+A_d(z)d  &=F_d(z)&&\text{in}&&(0,T)\times\Omega ,\\
\diver u &=0&&\text{in}&&(0,T)\times\Omega ,\\
u   &=0 &&\text{in}&&(0,T)\times\partial\Omega, \\
\mathcal{B}_{{d}}(\nabla)d   &=0 &&\text{in}&&(0,T)\times\partial\Omega, \\
(u,d)  &=(u_0,d_0) &&\text{in}&&\{0\}\times\Omega.
\end{aligned}\right.
\end{equation}
Here, for given
$$
\tilde{z}=(\tilde{u},\tilde{d})\in W_{p}^{2\mu-2/p}(\Omega;\RR^3)\times W_{p}^{2\mu+1-2/p}(\Omega;\RR^3),
$$
we set $-A_d(\tilde{z})d:=\cA_{\tilde{d}}(\nabla)d$, and
\begin{equation*}
\cA_{\tilde{d}}(\nabla)d:=2k_3\Delta d
+2(k_1-k_3)P_{\tilde{d}}\nabla\diver d+2(k_2-k_3)\left((\tilde{d}\times\nabla)\otimes \curl d\right)\cdot  \tilde{d},
\end{equation*}
is the principle part of the linearized Ericksen operator defined  in \eqref{eq:linEL}.
Furthermore,
\begin{align*}
\mathcal{B}_{\tilde{d}}(\nabla)d:
&=2k_3\nabla d\cdot \nu+2P_{\tilde{d}}(k_1\diver d\cdot I-k_3(\nabla d)^{\sf T})\cdot \nu
+2(k_2-k_3)(\tilde{d}\cdot \curl d)(\tilde{d}\times\nu)\\
&\quad +2(k_2+k_4)P_{\tilde{d}}((\nabla d)^{\sf T}-\diver d\cdot I)\cdot \nu,
\end{align*}
is defined as in \eqref{eq:BC2} and, following Section \ref{sec:BC_Er_Op}, we have
$$
\mathcal{B}_{{d}}(\nabla)d=P_d\frac{\partial\psi}{\partial(\nabla d)}\cdot\nu.
$$
For the definition of the second order operator $A_u(\tilde{z})$ and of the first order operators $R_1(\tilde{z})$ and  $R_0(\tilde{z})$ we refer to  \cite[pages 1454--1455]{HiPr19},
where, however,  one has to replace $\mathsf{u}$ by $\tilde{z}$ and $\nabla_w$ by $\nabla$ in that reference.
We note further  that $\diver(\nabla u)^{\sf T}=0$ since $\diver u=0$, so that the second term in the definition of $A_u(\tilde{z})$ in \cite{HiPr19} vanishes.
If $\tilde{z}=(\tilde{u},\tilde{d})\in\RR^3\times\RR^3$ is constant, a detailed expression of the Fourier-symbols of $R_j(\tilde{z})$ an $A_u(\tilde{z})$, is given  in the proof of
Theorem \ref{thm:MaxRegLin} below. Finally, $F_u(z)$ as well as $F_d(z)$ collect all terms of lower order.

\section{Properties of the linearized Ericksen operator}

The aim of this section is twofold: we first prove that in $\R^3$, the principal part of the linearized Ericksen operator is strongly elliptic in the sense of \cite[Section 6]{PrSi16}. Secondly, we show that in half-spaces, the linearized Ericksen operator subject to  the
principal part of the linearized boundary condition satisfies the Lopatinskii-Shapiro condition. The results obtained in this section are crucial for proving maximal regularity results in the following Section \ref{sec:LinMaxReg}.

We start by recalling  from  subsections \ref{sec:compEop} and \ref{sec:BC_Er_Op} that the Ericksen operator is of the form
\begin{align*}
  P_d\diver\left(\frac{\partial{\psi}}{\partial(\nabla d)}\right)&=2k_3\Delta d+2(k_1-k_3)P_d\nabla\diver d
  +2(k_2-k_3)\left((d\times\nabla)\otimes \curl d\right)\cdot  d\\
  &\quad +2(k_2-k_3)\left((d\times\nabla)\otimes d\right)\cdot \curl d -2(k_2-k_3)(d\cdot \curl d)P_d\curl d+2k_3|\nabla d|_2^2d
\end{align*}
and that its natural (nonlinear) boundary condition reads as
\begin{align*}
P_d\frac{\partial\psi}{\partial(\nabla d)}\cdot\nu&=2k_3\nabla d\cdot \nu+2P_d(k_1\diver d\cdot I-k_3(\nabla d)^{\sf T})\cdot \nu
+2(k_2-k_3)(d\cdot \curl d)(d\times\nu)\\
&\quad +2(k_2+k_4)P_d((\nabla d)^{\sf T}-\diver d\cdot I)\cdot \nu.
\end{align*}
We recall from Section 3 that for sufficiently smooth $\eta$, the principle part of the linearized Ericksen operator is given  by
\begin{equation}\label{eq:Ell1}
\mathcal{A}_d(\nabla)\eta:=2k_3\Delta \eta+2(k_1-k_3)P_d\nabla\diver \eta+2(k_2-k_3)\left((d\times\nabla)\otimes \curl \eta\right)\cdot  d,
\end{equation}
and that the principle part of the linearized boundary condition reads as
\begin{align}\label{eq:BC3}
\begin{split}
\mathcal{B}_d(\nabla)\eta&:=2k_3\nabla \eta\cdot \nu+2P_d(k_1\diver \eta\cdot I-k_3(\nabla \eta)^{\sf T})\cdot \nu
+2(k_2-k_3)(d\cdot \curl \eta)(d\times\nu)\\
&\quad +2(k_2+k_4)P_d((\nabla \eta)^{\sf T}-\diver \eta\cdot I)\cdot \nu,
\end{split}
\end{align}
with constant coefficients $d\in\RR^3$ and with $|d|_2=1$.

\subsection{Strong Ellipticity}\mbox{}

We show in the following that the operator $-\cA_d(\nabla)$ given by \eqref{eq:Ell1} is strongly elliptic. To this end, we denote by
\begin{align*}
m_d(\xi)\eta&=2k_3|\xi|_2^2\eta+2(k_1-k_3)P_d(\xi\otimes\xi)\eta+2(k_2-k_3)\left((d\times\xi)\otimes (\xi\times\eta)\right)\cdot d\\
&=2k_3|\xi|_2^2\eta+2(k_1-k_3)P_d(\xi\otimes\xi)\eta+2(k_2-k_3)(d\times\xi)[(\xi\times\eta)\cdot d]\\
&=2k_3|\xi|_2^2\eta+2(k_1-k_3)P_d(\xi\otimes\xi)\eta+2(k_2-k_3)(d\times\xi)[(d\times\xi)\cdot \eta]\\
&=\left[2k_3|\xi|_2^2I+2(k_1-k_3)P_d(\xi\otimes\xi)+2(k_2-k_3)(d\times\xi)\otimes(d\times\xi)\right]\eta
\end{align*}
the Fourier symbol of the operator $-\cA_d(\nabla)$.

\begin{proposition}[Strong ellipticity]\label{prop:elliptic}
Assume (F) and let $d\in\RR^3$ with $|d|_2=1$. Then the operator $-\cA_d(\nabla)$ is strongly elliptic, i.e., there exists a constant $c>0$ such that
\begin{equation}\label{eq:normell}
m_d(\xi)\eta\cdot\eta\ge c|\xi|_2^2|\eta|_2^2
\end{equation}
for all $\xi,\eta\in\RR^3$.
\end{proposition}

\begin{proof} We consider in a first step the condition $9k_3>k_1$ from assumption (F). Let us denote by
$$
m_d^{sym}(\xi)=2k_3|\xi|_2^2I+(k_1-k_3)[P_d(\xi\otimes\xi)+(\xi\otimes\xi)P_d]+2(k_2-k_3)(d\times\xi)\otimes(d\times\xi)
$$
the symmetric part of $m_d(\xi)$.
We note that for $\xi\in\RR^3$, $\xi\neq 0$, it holds that $m_d^{sym}(\xi)=|\xi|_2^2\tilde{m}_d^{sym}(\zeta)$, where $\zeta=\xi/|\xi|_2$ and
$$\tilde{m}_d^{sym}(\zeta)=2k_3I+(k_1-k_3)[P_d(\zeta\otimes\zeta)+(\zeta\otimes\zeta)P_d]+2(k_2-k_3)(d\times\zeta)\otimes(d\times\zeta).$$
We will prove that the symmetric matrix $\tilde{m}_d^{sym}(\zeta)\in\RR^{3\times 3}$ is positive definite. To this end, we distinguish several cases.

Case I: Assume that $d\ \|\ \zeta$. Then $d\times \zeta=0$ and $d=\zeta$ or $d=-\zeta$, since $|d|_2=1$. It follows that $\zeta\otimes\zeta=d\otimes d$ and therefore
$$(I-d\otimes d)(d\otimes d)=(d\otimes d)(I-d\otimes d)=(d\otimes d)-(d\otimes d)=0,$$
where we made again use of the fact $|d|_2=1$. So, in this case,
$$\tilde{m}_d^{sym}(\zeta)=2k_3I,$$
wherefore all three eigenvalues coincide with $2k_3>0$.

Case II: Assume that $d\times\zeta\neq 0$. In this case, $d\times\zeta$ is an eigenvector of $\tilde{m}_d^{sym}(\zeta)$, since
\begin{align*}
\tilde{m}_d^{sym}(\zeta)(d\times \zeta)&=2k_3(d\times\zeta)+2(k_2-k_3)|d\times\zeta|_2^2(d\times\zeta)\\
&=\left(2k_2|d\times\zeta|_2^2+2k_3(1-|d\times\zeta|_2^2)\right)(d\times\zeta).
\end{align*}
Here we have used the fact that $\zeta,d\perp (d\times\zeta)$. The corresponding eigenvalue satisfies
$$2k_2|d\times\zeta|_2^2+2k_3(1-|d\times\zeta|_2^2)\ge2\min\{k_2,k_3\}>0,$$
since $|d|_2=|\zeta|_2=1$.
We are now looking for other eigenvalues and eigenvectors and make the ansatz $\alpha\zeta+\beta d$, $\alpha,\beta\in\RR$. Since
$$(\alpha\zeta+\beta d)\perp (d\times\zeta)$$
a short computation shows that
\begin{align*}
\tilde{m}_d^{sym}(\zeta)(\alpha\zeta+\beta d)&=[2k_3\alpha+(k_1-k_3)(\alpha+\beta z+\alpha (1-z^2))]\zeta\\
&+ [2k_3\beta-(k_1-k_3)(\alpha z+\beta z^2)]d,
\end{align*}
where $z:=(d|\zeta)$.

Case II.1: If $z=0$, then $\zeta$ is an eigenvector with eigenvalue $2k_1>0$ ($(\alpha,\beta)=(1,0)$) and also $d$ is an eigenvector with eigenvalue $2k_3>0$ ($(\alpha,\beta)=(0,1)$).

Case II.2: In case $z\neq 0$, we require that there exists $\lambda$ (an eigenvalue) such that
$$2k_3\alpha+(k_1-k_3)(\alpha+\beta z+\alpha (1-z^2))=\lambda\alpha$$
and
$$2k_3\beta-(k_1-k_3)(\alpha z+\beta z^2)=\lambda\beta.$$
If $k_1=k_3$, then $\zeta$ and $d$ are eigenvectors to the same eigenvalue $2k_3$. In case $k_1\neq k_3$, we solve the last equation for $\lambda$, to obtain (observe that $\beta=0$ would yield $\alpha=0$ in case $z\neq 0$ and $k_1\neq k_3$)
\begin{equation}\label{eq:EigValLambda}
\lambda=2k_3-(k_1-k_3)(\alpha z/\beta+z^2),
\end{equation}
and plug it into the other equation. This yields
$$2\alpha+\beta z+\alpha^2 z/\beta=0.$$
Multiplying with $\beta$ and dividing by $z\neq 0$ implies
$$\alpha^2+\frac{2}{z}\alpha\beta+\beta^2=0.$$
This quadratic equation can be solved for $\alpha$ to the result
$$\alpha=-\frac{\beta}{z}(1\pm\sqrt{1-z^2}),$$
which in turn implies
$$\frac{\alpha z}{\beta}=-1\mp\sqrt{1-z^2}.$$
We insert this expression into the equation \eqref{eq:EigValLambda} for $\lambda$, to obtain
\begin{align*}
\lambda_\pm&=2k_3-(k_1-k_3)(-1\mp\sqrt{1-z^2}+z^2)\\
&=2k_3+(k_1-k_3)(1-z^2\pm\sqrt{1-z^2}).
\end{align*}
This shows that, if $k_1<k_3$, then $\lambda_+<\lambda_-$ and
$$\lambda_+= 2k_3+(k_1-k_3)(1-z^2+\sqrt{1-z^2})\ge 2k_3+2(k_1-k_3)=2k_1>0,$$
since
$$0\le 1-z^2+\sqrt{1-z^2}\le 2$$
for $|z|=|(d|\zeta)|\le 1$, by the Cauchy-Schwarz inequality.

On the contrary, if $k_1>k_3$, then $\lambda_+>\lambda_-$ and
$$\lambda_-= 2k_3+(k_1-k_3)(1-z^2-\sqrt{1-z^2}),$$
This time we use the estimate
$$1-z^2-\sqrt{1-z^2}\ge -\frac{1}{4},$$
valid for all $|z|\le 1$, to obtain
$$\lambda_-\ge 2k_3-\frac{1}{4}(k_1-k_3)=\frac{9k_3-k_1}{4}>0,$$
by the assumption $9k_3>k_1$ from (F). This shows that the symmetric matrix $\tilde{m}_d^{sym}(\zeta)$ is positive definite, which in turn is equivalent to the fact that the matrix $\tilde{m}_d(\zeta)$ is positive definite.

Let us finally consider the condition $2|k_1-k_3|<\min\{k_2,k_3\}$ from assumption (F). The considerations from above for the very special case $k_1=k_3$ show that the symmetric matrix
$$
M_d(\xi):=2k_3|\xi|_2^2I+2(k_2-k_3)(d\times\xi)\otimes(d\times\xi)
$$
is positive definite and satisfies the estimate
$$M_d(\xi)\eta\cdot\eta\ge 2\min\{k_2,k_3\}|\xi|_2^2|\eta|_2^2.$$
Since $|d|_2=1$, it holds that
$$2|P_d(\xi\otimes\xi)\eta\cdot\eta|=2|(I-d\otimes d)(\xi\otimes\xi)\eta\cdot\eta|\le 4|\xi|_2^2|\eta|_2^2$$
by the Cauchy-Schwarz inequality, which implies that the matrix
$$M_d(\xi)+2(k_1-k_3)P_d(\xi\otimes\xi)$$
is positive definite, provided $2|k_1-k_3|<\min\{k_2,k_3\}$.
This completes the proof of Proposition \ref{prop:elliptic}.
\end{proof}

\begin{remark}
{\rm If, instead of condition (F), one merely assumes that the Frank coefficients satisfy $k_j>0\ \text{for}\ j\in\{1,2,3\}$,
one can prove that all eigenvalues of the (non-symmetric) matrix-valued symbol $m_d(\xi)$ of $-\cA_d(\nabla)$ are real and positive, i.e. the operator $-\cA_d(\nabla)$ is \emph{normally elliptic}.  Let us emphasize that, in general, normal ellipticity \emph{does not} imply strong ellipticity. However, in the proof of Theorem \ref{thm:MaxRegLin} below we need to know that $-\cA_d(\nabla)$ is strongly elliptic.}
\end{remark}

\subsection{The Lopatinskii-Shapiro condition}\mbox{}

For given $d\in\RR^3$ with $|d|_2=1$, we consider the pair $(\cA_d(\nabla),\cB_d(\nabla))$, defined in \eqref{eq:Ell1} and \eqref{eq:BC3}.
Our aim is to prove that $(\cA_d(\nabla),\cB_d(\nabla))$ satisfies the Lopatinskii-Shapiro condition, formulated precisely in  Proposition \ref{pro:LopShap_d} below. To this end,
let $\sH$ denote a half space in $\RR^3$, hence $\sH$ may be written as a translation and/or rotation of the half space $\RR_+^3$. We start with the following

\begin{proposition}\label{prop:LSC0}
Assume (F) and let $d\in\RR^3$ with $|d|_2=1$.
Then, for each $\lambda\in\CC$ with $\real\lambda\ge 0$, the only solution $\eta\in H_2^2(\sH;\CC^3)$ of the boundary value problem
$$\lambda\eta-\cA_d(\nabla)\eta=0\quad\text{in $\sH$},\quad \cB_d(\nabla) \eta=0\quad\text{on $\partial\sH$},$$
is $\eta=0$. Moreover, for each $\eta_1\in H_2^2(\sH;P_d\CC^3)$ with $\cB_d(\nabla) \eta_1=0$, it holds that
$$
-(\mathcal{A}_d(\nabla)\eta_1|\eta_1)_{L_2}\ge\alpha\|\nabla\eta_1\|_{L_2(\sH)}^2,
$$
with the constant $\alpha$ from assumption (F).
\end{proposition}

\begin{proof}
We split $\eta$ and write $\eta=P_d\eta+(I-P_d)\eta=:\eta_1+\eta_2$. Applying $(I-P_d)$ to the equation $\lambda\eta-\cA_d(\nabla)\eta=0$ yields
$$
\lambda\eta_2-2k_3(I-P_d)\Delta(\eta_1+\eta_2)=0,
$$
since $(I-P_d)P_d=0$, as $|d|_2=1$, and $\left((d\times\nabla)\otimes \curl \eta\right)\cdot  d\perp d$.
Moreover, we have
$$(I-P_d)\Delta(\eta_1+\eta_2)=\Delta(I-P_d)(\eta_1+\eta_2)=\Delta\eta_2.$$
Applying $(I-P_d)$ to the boundary condition $\cB_d(\nabla)\eta=0$ yields
$$0=2k_3(I-P_d)\nabla (\eta_1+\eta_2)\cdot\nu=2k_3\nabla (I-P_d)(\eta_1+\eta_2)\cdot\nu=2k_3\nabla\eta_2\cdot\nu.$$
Therefore, $\eta_2$ solves the problem
$$\lambda\eta_2-2k_3\Delta\eta_2=0\quad\text{in $\sH$},\quad \nabla\eta_2\cdot\nu=0\quad\text{on $\partial\sH$}.$$
Taking the inner product with $\eta_2$ in $L_2(\sH;\CC^3)$, integrating by parts and taking real parts, yields
$$\real\lambda\|\eta_2\|_{L_2(\sH)}^2+2k_3\|\nabla\eta_2\|_{L_2(\sH)}^2=0.$$
If $\real\lambda>0$, then $\eta_2=0$. If $\real\lambda=0$, then $\eta_2$ is constant, hence $\eta_2=0$, since otherwise $\eta_2\notin L_2(\sH)$. This shows that $\eta=\eta_1=P_d\eta$.

Taking the inner product of $-\cA_d(\nabla)\eta_1$ with $\eta_1$  in $L_2(\sH;\CC^3)$, it follows that
\begin{equation*}
(\cA_d(\nabla)\eta_1|\eta_1)_{L_2}=2k_3(\Delta\eta_1|\eta_1)_{L_2}\\+2(k_1-k_3)(\nabla\diver\eta_1|\eta_1)_{L_2}+2(k_2-k_3)(\left((d\times\nabla)\otimes \curl \eta_1\right)\cdot  d|\eta_1)_{L_2},
\end{equation*}
since $P_d^{\sf T}=P_d$ and $P_d\eta_1=\eta_1$.
We write
$$k_1(\nabla\diver\eta_1|\eta_1)_{L_2}=k_1(\diver(\diver\eta_1\cdot I)|\eta_1)_{L_2}$$
and
$$k_3(\nabla\diver\eta_1|\eta_1)_{L_2}=k_3(\diver(\nabla\eta_1)^{\sf T}|\eta_1)_{L_2}.$$
Since $d\in\RR^3$ is constant, the latter term can be rewritten as
$$
(\left((d\times\nabla)\otimes \curl \eta_1\right)\cdot  d|\eta_1)_{L_2}=(\diver(\mathsf{D}(d\cdot\curl\eta_1))|\eta_1)_{L_2},
$$
where
$$
\mathsf{D}:=
\begin{pmatrix}
0 & -d_3 & d_2\\
d_3 & 0 & -d_1\\
-d_2 & d_1 & 0
\end{pmatrix}.
$$
Integrating by parts and invoking the boundary condition $\cB_d(\nabla) \eta_1=0$, yields
\begin{align*}
2k_3(\Delta\eta_1|\eta_1)_{L_2}&+2k_1(\diver(\diver\eta_1\cdot I)|\eta_1)_{L_2}-2k_3(\diver(\nabla\eta_1)^{\sf T}|\eta_1)_{L_2}
+2(k_2-k_3)(\diver(\mathsf{D}(d\cdot\curl\eta_1))|\eta_1)_{L_2}\\
                               &=-2k_3(\nabla\eta_1|\nabla\eta_1)_{L_2}-2k_1\|\diver\eta_1\|_{L_2}^2+2k_3((\nabla\eta_1)^{\sf T}|\nabla\eta_1)_{L_2} \\
                               & \quad  -2(k_2-k_3)(\mathsf{D}(d\cdot\curl\eta_1)|\nabla\eta_1)_{L_2}
                                 -2(k_2+k_4)((\nabla \eta_1)^{\sf T}-\diver \eta_1\cdot I)\cdot \nu|\eta_1)_{L_{2,\partial\sH}}
\end{align*}
Another integration by parts yields, that the last (boundary) term can be written as
\begin{align*}
  ((\nabla \eta_1)^{\sf T}-\diver \eta_1\cdot I)\cdot \nu|\eta_1)_{L_{2,\partial\sH}}&=((\nabla \eta_1)^{\sf T}-\diver \eta_1\cdot I)|\nabla\eta_1)_{L_{2}}
                                                                                       = ((\nabla \eta_1)^{\sf T}|\nabla \eta_1)_{L_2}-\|\diver \eta_1\|_{L_{2}}^2\\
&=((\nabla \eta_1)^{\sf T}-\nabla\eta_1|\nabla \eta_1)_{L_2}+(\nabla \eta_1|\nabla \eta_1)_{L_2}-\|\diver \eta_1\|_{L_{2}}^2,
\end{align*}
since $\diver (\nabla\eta_1)^{\sf T}-\diver(\diver\eta_1\cdot I)=0$.

A short computation shows that $(\mathsf{D}(d\cdot\curl\eta_1)|\nabla\eta_1)_{L_2}=\|d\cdot\curl \eta_1\|_{L_2}^2$. In summary, this yields
\begin{align*}
-(\cA_d(\nabla)\eta_1|\eta_1)_{L_2}&= 2k_1\|\diver\eta_1\|_{L_2}^2+2k_3(\nabla\eta_1-(\nabla\eta_1)^{\sf T}|\nabla\eta_1)_{L_2} +2(k_2-k_3)\|d\cdot\curl \eta_1\|_{L_2}^2\\
&\quad +2(k_2+k_4)[((\nabla \eta_1)^{\sf T}-\nabla\eta_1|\nabla \eta_1)_{L_2}+\|\nabla \eta_1\|_{L_2}^2-\|\diver \eta_1\|_{L_{2}}^2]
\end{align*}
We observe that
$$(\nabla\eta_1-(\nabla\eta_1)^{\sf T}|\nabla\eta_1)_{L_2}=\|\curl\eta_1\|_{L_2}^2,$$
whence
\begin{align*}
  -(\cA_d(\nabla)\eta_1|\eta_1)_{L_2}&= 2(k_1-k_2-k_4)\|\diver\eta_1\|_{L_2}^2+2(k_3-k_2-k_4)\|\curl\eta_1\|_{L_2}^2\\
 &\quad  +2(k_2-k_3)\|d\cdot\curl \eta_1\|_{L_2}^2+(k_2+k_4)\|\nabla \eta_1\|_{L_2}^2.
\end{align*}
Since $|d|_2=1$, we may write
$$
\|\curl\eta_1\|_{L_2}^2=\|d\times\curl\eta_1\|_{L_2}^2+\|d\cdot\curl\eta_1\|_{L_2}^2,
$$
which in turn yields
\begin{align*}
-(\cA_d(\nabla)\eta_1|\eta_1)_{L_2}&= 2(k_1-k_2-k_4)\|\diver\eta_1\|_{L_2}^2
+2(k_3-k_2-k_4)\|d\times\curl\eta_1\|_{L_2}^2\\
&\quad -k_4\|d\cdot\curl \eta_1\|_{L_2}^2+(k_2+k_4)\|\nabla \eta_1\|_{L_2}^2.
\end{align*}
For any $0<\alpha\le\min\{k_1,k_2,k_3\}$, it follows from (F) that $k_4=\alpha-k_2$, hence
$$
k_j-k_2-k_4=k_j-\alpha\ge 0,\ j\in\{1,3\},\quad \mbox{and} \quad -k_4\ge 0,
$$
which implies
$$-(\cA_d(\nabla)\eta_1|\eta_1)_{L_2}\ge \alpha\|\nabla \eta_1\|_{L_2(\sH)}^2.$$
This is the claimed estimate.

Finally, taking the inner product of $\lambda\eta_1-\cA_d(\nabla)\eta_1=0$ with $\eta_1$ in $L_2(\sH;\CC^3)$, and taking real parts, we obtain
$$
0=\real\lambda\|\eta_1\|_{L_2(\sH)}^2-(\cA_d(\nabla)\eta_1|\eta_1)_{L_2}\ge \real\lambda\|\eta_1\|_{L_2(\sH)}^2+\alpha\|\nabla \eta_1\|_{L_2(\sH)}^2,
$$
which in turn implies that $\eta_1=0$ .
\end{proof}

By the same strategy, we are able to prove the following result.

\begin{proposition}[Lopatinskii-Shapiro condition]\label{pro:LopShap_d}
Assume (F) and let $d\in\RR^3$ with $|d|_2=1$.
Then, for all $\lambda\in\CC$ with $\real\lambda\ge 0$ and all $\xi,\nu\in\RR^3$ with $\xi\perp\nu$ and $(\lambda,\xi)\neq (0,0)$, the only solution $\eta\in H_2^2(\RR_+;\CC^3)$ of the initial value problem
$$\lambda\eta-\cA_d(i\xi+\nu\partial_y)\eta=0,\quad y>0,\quad \cB_d(i\xi+\nu\partial_y) \eta=0,\quad y=0,$$
is $\eta=0$. Moreover, for each $\eta_1\in H_2^2(\RR_+;P_d\CC^3)$ with $\cB_d(i\xi+\nu\partial_y) \eta_1=0$, it holds that
$$-(\mathcal{A}_d(i\xi+\nu\partial_y)\eta_1|\eta_1)_{L_2}\ge\alpha(|\xi|^2\|\eta_1\|_{L_2(\RR_+)}^2+\|\partial_y\eta_1\|_{L_2(\RR_+)}^2),$$
with the constant $\alpha$ from assumption (F).
\end{proposition}

\begin{proof}
We use the same strategy as in the proof of the preceding proposition, however now  the differential operator $\nabla$ is being replaced by $i\xi+\nu\partial_y$. The splitting $\eta=\eta_1+\eta_2$ with $\eta_1=P_d\eta$ and integration by parts with respect to the variable $y>0$ yields
$$\real\lambda\|\eta_2\|_{L_2(\RR_+)}^2+2k_3(|\xi|^2\|\eta_2\|_{L_2(\RR_+)}^2+\|\partial_y\eta_2\|_{L_2(\RR_+)}^2)=0,$$
which shows that $\eta_2=(I-P_d)\eta=0$, hence $\eta=\eta_1=P_d\eta$. Taking the inner product of $-\cA_d(i\xi+\nu\partial_y)\eta_1$ with $\eta_1$ in $L_2(\RR_+)^3$ and integrating by parts with respect to $y>0$, implies the desired estimate. In fact, the technical steps from the proof of Proposition \ref{prop:LSC0} can be mimicked and are therefore omitted.
Finally, taking the inner product of the equation $\lambda\eta_1-\cA_d(i\xi+\nu\partial_y)\eta_1=0$ with $\eta_1$ in $L_2(\RR_+;\CC^3)$ yields $\eta_1=0$, hence $\eta=0$.
\end{proof}

\begin{remark}\label{rem:LS}{\rm
Since the state space $\CC^3$ of $\eta$ in Proposition \ref{pro:LopShap_d} is finite dimensional, strong ellipticity of $-\cA_d(\nabla)$ implies that the Lopatinskii-Shapiro condition is equivalent to the fact that for each $g\in\CC^3$, the problem
$$\lambda\eta-\cA_d(i\xi+\nu\partial_y)\eta=0,\quad y>0,\quad \cB_d(i\xi+\nu\partial_y) \eta=g,\quad y=0$$
has a unique solution $\eta\in H_q^2(\RR_+;\CC^3)$ for \emph{any} $1<q<\infty$. We refer to  \cite[Remark 6.2.2 (iv) \& Section 6.2.2]{PrSi16} for details.
}
\end{remark}

\section{Linearized problems and maximal regularity}\label{sec:LinMaxReg}

\noindent
We start by considering  a linearized problem for the function $d$, which reads as
\begin{equation}\label{eq:lin_d}
\left\{\begin{aligned}
\gamma\partial_t d-\mathcal{A}_{\tilde{d}}(\nabla)d &= f &&\text{in}&&(0,T)\times\Omega ,\\
\mathcal{B}_{\tilde{d}}(\nabla)d   &=g &&\text{in}&&(0,T)\times\partial\Omega, \\
d  &=d_0 &&\text{in}&&\{0\}\times\Omega,
\end{aligned}\right.
\end{equation}
where $\mathcal{A}_{\tilde{d}}(\nabla)d$ and $\mathcal{B}_{\tilde{d}}(\nabla)d$ are defined as in Section \ref{sec:FAsett}.

\begin{proposition}\label{thm:dEq}
Let  $1<p<\infty$, $\mu \in (\frac12 + \frac{5}{2p},1]$, $\gamma>0$ and assume (F). Suppose that $T>0$, $\Omega\in\{\RR^3,\RR^3_+\}$ and let $\tilde{d}\in \RR^3$ with $|\tilde{d}|_2=1$.
Then, for any $T>0$, the problem
\eqref{eq:lin_d}
admits a unique solution
$$d\in H_{p,\mu}^1((0,T);H_p^1(\Omega;\RR^3))\cap L_{p,\mu}((0,T);H_p^{3}(\Omega;\RR^3)),$$
if and only if the data are subject to the following conditions.
\begin{enumerate}
\item $f\in L_{p,\mu}((0,T);H_p^1(\Omega;\RR^3))$,
\item $g\in  W_{p,\mu}^{1-1/2p}((0,T);L_p(\partial\Omega;\RR^3))\cap L_{p,\mu}((0,T);W_{p}^{2-1/p}(\partial\Omega;\RR^3))$,
\item $d_0\in W_{p}^{2\mu+1-2/p}(\Omega;\RR^3)$
\item $\mathcal{B}_{\tilde{d}}(\nabla)d_0=g(t=0)$.
\end{enumerate}
\end{proposition}

\begin{proof}
The assertion follows from \cite[Section 6.1.5 (i) \& Proof of Theorem 6.3.3]{PrSi16}, since $-\mathcal{A}_{\tilde{d}}(\nabla)$ is strongly elliptic by Proposition \ref{prop:elliptic} and since the pair $(\mathcal{A}_{\tilde{d}}(\nabla),\mathcal{B}_{\tilde{d}}(\nabla))$ satisfies the Lopatinskii-Shapiro condition by Proposition \ref{pro:LopShap_d}.
\end{proof}

Next, we linearize the problem \eqref{eq:iEL2} for $(u,d)$ at some given
$$
\tilde{z}=(\tilde{u},\tilde{d})\in W_{p}^{2\mu-2/p}(\Omega;\RR^3)\times W_{p}^{2\mu+1-2/p}(\Omega;\RR^3),
$$
and drop all terms of lower order. This yields the \emph{principal linearization} of equation  \eqref{eq:iEL2}
\begin{equation}\label{eq:princLin}
\left\{\begin{aligned}
\rho\partial_tu+A_u(\tilde{z})u +R_1^{\sf T}(\tilde{z})\partial_t d+\nabla \pi &=f_u&&\text{in}&&(0,T)\times\Omega ,\\
-R_0(\tilde{z})u+\gamma(\tilde{z})\partial_td+A_d(\tilde{z})d  &=f_d&&\text{in}&&(0,T)\times\Omega ,\\
\diver u &=0&&\text{in}&&(0,T)\times\Omega ,\\
u   &=0 &&\text{in}&&(0,T)\times\partial\Omega, \\
\mathcal{B}_{\tilde{d}}(\nabla)d   &=g &&\text{in}&&(0,T)\times\partial\Omega, \\
(u,d)  &=(u_0,d_0) &&\text{in}&&\{0\}\times\Omega.
\end{aligned}\right.
\end{equation}
For the system \eqref{eq:princLin}, we prove the following maximal regularity result.

\begin{theorem}\label{thm:MaxRegLin}
Let  $1<p<\infty$, $\mu \in (\frac12 + \frac{5}{2p},1]$ and assume (F),(P) and (R). Let $\Omega\subset\R^3$ be a bounded domain with boundary $\partial\Omega\in C^{3}$ and let
$$\tilde{z}=(\tilde{u},\tilde{d})\in W_{p}^{2\mu-2/p}(\Omega;\RR^3)\times W_{p}^{2\mu+1-2/p}(\Omega;\RR^3),$$
with $|\tilde{d}(x)|_2=1$, $x\in{\Omega}$.
Then, for any $T\in (0,\infty)$, the problem \eqref{eq:princLin} admits a unique solution
\begin{align*}
  u&\in H_{p,\mu}^1((0,T);L_p(\Omega;\RR^3))\cap L_{p,\mu}((0,T);H_p^{2}(\Omega;\RR^3)), \\
\nabla\pi &\in L_{p,\mu}((0,T);L_p(\Omega;\RR^3)),\\
d &\in H_{p,\mu}^1((0,T);H_p^1(\Omega;\RR^3))\cap L_{p,\mu}((0,T);H_p^{3}(\Omega;\RR^3)),
\end{align*}
if and only if the data are subject to the following conditions.
\begin{enumerate}
\item $f_u\in L_{p,\mu}((0,T);L_p(\Omega;\RR^3))$,
\item $f_d\in L_{p,\mu}((0,T);H_p^1(\Omega;\RR^3))$,
\item $g\in  W_{p,\mu}^{1-1/2p}((0,T);L_p(\partial\Omega;\RR^3))\cap L_{p,\mu}((0,T);W_{p}^{2-1/p}(\partial\Omega;\RR^3))$,
\item $u_0\in W_{p}^{2\mu-2/p}(\Omega;\RR^3)$,
\item $\diver u_0=0$,
\item $u_0=0$ on $\partial\Omega$,
\item $d_0\in W_{p}^{2\mu+1-2/p}(\Omega;\RR^3)$,
\item $\mathcal{B}_{\tilde{d}}(\nabla)d_0=g(t=0)$ on $\partial\Omega$.
\end{enumerate}
\end{theorem}

\begin{proof}
The proof is subdivided into several steps.

\vspace{.1cm}
\noindent
\emph{Step1:} Let $\Omega=\RR^3$ and consider the case of constant coefficients, i.e., we linearize at a constant vector $\tilde{z}=(\tilde{u},\tilde{d})\in\RR^3\times\RR^3$
such that $|\tilde{d}|_2=1$. By Proposition \ref{thm:dEq} for $\Omega=\RR^3$ and \cite[Theorem 7.1.1]{PrSi16}, we may first reduce \eqref{eq:princLin} to the case $(u_0,d_0)=(0,0)$.  Note that the inhomogeneities $(f_u,f_d)$ then have to replaced by some modified data in the right regularity classes.

Denoting by $\xi$ the Fourier variable in space and by $\lambda$ the Laplace variable in time, the Laplace-Fourier symbol of the differential operator defined by the left side of $\eqref{eq:princLin}_1-\eqref{eq:princLin}_3$ is given by
\begin{equation*}\label{princ-symb}
\cL_{(\tilde{u},\tilde{d})}(\lambda,i\xi) = \left( \begin{array}{cccc}
{M_u(\lambda,\xi)}&i\xi^{\sf T}  &  {i\lambda R_1(\xi)^{\sf T}}\\
i\xi & 0&0\\
{-iR_0(\xi)}&0& M_{d}(\lambda,\xi)
\end{array}\right),
\end{equation*}
where $iR_j(\xi)$ are the Fourier-symbols of the first order differential operators $R_j(\tilde{z})$, $M_u(\lambda,\xi)$ is the Laplace-Fourier-symbol of
$$
\rho\partial_t+A_u(\tilde{z}),
$$
and $M_d(\lambda,\xi)$ is the Laplace-Fourier-symbol of
$$
\gamma(\tilde{z})\partial_t+A_d(\tilde{z})=\gamma(\tilde{z})\partial_t-\cA_{\tilde{d}}(\nabla).
$$
In particular, it holds that
\begin{align*}
M_{{d}}(\lambda,\xi) &= \gamma \lambda I+ m_{\tilde{d}}(\xi),\\
 m_{\tilde{d}}(\xi) &= 2k_3|\xi|^2I+2(k_1-k_3)(I-\tilde{d}\times \tilde{d})\xi\otimes\xi+2(k_2-k_3)(\tilde{d}\times\xi)\otimes (\tilde{d}\times\xi),\\
 R_0(\xi)&= \frac{\mu_D +\mu_V}{2} P_{\tilde{d}}\xi\otimes \tilde{d} + \frac{\mu_D -\mu_V}{2}(\xi| \tilde{d}) P_{\tilde{d}},\\
 R_1(\xi)& = R_\mu(\xi) -R_0(\xi),\\
 R_\mu(\xi)&= \mu_+P_{\tilde{d}}\xi\otimes \tilde{d}+\mu_-(\xi|\tilde{d}) P_{\tilde{d}},\\
  M_u(\lambda,\xi)&= m_u(\lambda,\xi)I + \mu_0(\xi|\tilde{d})^2\tilde{d}\otimes \tilde{d} + \frac{\mu_L}{4} R(\xi)^{\sf T}R(\xi) +
                    \frac{1}{4\gamma} R_\mu(\xi)^{\sf T}R_\mu (\xi)+ \frac{\mu_P\mu_V}{2\gamma}(\xi|\tilde{d})(R(\xi)-R^{\sf T}(\xi)),\\
R(\xi)&=  P_{\tilde{d}}\xi\otimes \tilde{d}+(\xi|\tilde{d}) P_{\tilde{d}},\\
m_u(\lambda,\xi) &=\rho \lambda +\mu_s|\xi|^2,
\end{align*}
where $P_{\tilde{d}}=I-\tilde{d}\otimes \tilde{d}$ and  $\mu_\pm = \mu_D \pm \mu_V+\mu_P$. Let us remark that, apart from $M_d(\lambda,\xi)$, the definitions of the above symbols
coincide with those in Section 5 of \cite{HiPr17} or \cite{HiPr19} .

For the time being consider, the purely parabolic part $\mathcal{L}^0$ of the symbol $\mathcal{L}$
\begin{equation*}\label{princ-symb_parab}
\cL_{(\tilde{u},\tilde{d})}^0(z,i\xi) = \left(\begin{array}{cccc}
{M_u(\lambda,\xi)}&  {i\lambda R_1(\xi)^{\sf T}}\\
{-iR_0(\xi)}& M_{{d}}(\lambda,\xi)
\end{array}\right),
\end{equation*}
which results from $\mathcal{L}_{(\tilde{u},\tilde{d})}$ by dropping the pressure gradient and divergence equation.
For $J(u,d):=(u,\lambda d)$ and $v:=(u,d)$, a computation shows that
\begin{align*}
  \real(\mathcal{L}_{(\tilde{u},\tilde{d})}^0v|Jv)&=\real m_u(\lambda,\xi) |u|_2^2 + \mu_0(\xi|\tilde{d})^2|(\tilde{d}|u)|^2 +\frac{\mu_L}{4} |Ru|_2^2 \\
 & \quad  + \frac{1}{4\gamma} |R_\mu u|_2^2   + \real[ i \lambda (d|R_\mu u)] + \gamma|\lambda|^2|d|_2^2+\real \lambda(m_{\tilde{d}}(\xi)d|d).
\end{align*}
By strong ellipticity of $-\cA_{\tilde{d}}(\nabla)$, see \eqref{eq:normell}, we obtain
$$(m_{\tilde{d}}(\xi)d|d)\ge c|\xi|_2^2|d|_2^2$$
for some constant $c>0$. Furthermore,
\begin{align*}
\frac{1}{4\gamma} |R_\mu u|_2^2   + \real[ i \lambda (d|R_\mu u)] + \gamma|\lambda|^2|d|_2^2&\ge
\frac{1}{4\gamma} |R_\mu u|_2^2   -|\lambda| |d|_2|R_\mu u|_2 + \gamma|\lambda|^2|d|_2^2\\
&=\big(\sqrt{\gamma}|\lambda||d|_2-\frac{1}{2\sqrt{\gamma}}|R_\mu u|_2\big)^2,
\end{align*}
by the Cauchy-Schwarz inequality. Assumption (P) then yields the estimate
\begin{equation}\label{eq:AccretEst}
\real(\mathcal{L}_{(\tilde{u},\tilde{d})}^0v|Jv)\ge (\rho\real \lambda+\mu_s|\xi|_2^2)|u|_2^2,
\end{equation}
provided $\real\lambda\ge 0$. Next, we consider  the equation
$$-iR_0(\xi)u+M_{{d}}(\lambda,\xi)d=f_d$$
and solve it for $d$. Let us note that for $\real \lambda\ge 0$ and $(\lambda,\xi)\neq (0,0)$, the matrix $M_{{d}}(\lambda,\xi)$ is invertible by \eqref{eq:normell}. Therefore, we obtain
$$d=M_{{d}}(\lambda,\xi)^{-1}\left(f_d+iR_0(\xi)u\right).$$
For $f_d=0$, this yields the \emph{Schur complement}
$$M(\lambda,i\xi):=M_u(\lambda,\xi)-\lambda R_1(\xi)^{\sf T}M_{{d}}(\lambda,\xi)^{-1}R_0(\xi)$$
for $u$ and \eqref{eq:AccretEst} implies the estimate
$$\real (M(\lambda,i\xi)u|u)\ge (\rho\real \lambda+\mu_s|\xi|_2^2)|u|_2^2,$$
since Schur reduction preserves positive (semi) definiteness.
Therefore, we are in a position to apply the techniques from \cite[Section 7.1]{PrSi16} to prove maximal $L_p$-regularity of the corresponding generalized Stokes problem for $u$ in the full space $\RR^3$. In fact, in
\cite[Section 7.1]{PrSi16} one has to replace $\lambda+\mathcal{A}(\xi)$ by  $M(\lambda,i\xi)$.

Having a unique solution $u$ of the generalized Stokes equation in its optimal regularity class, it follows that
$$R_0(\tilde{z})u\in L_{p,\mu}((0,T);H_p^1(\RR^3;\RR^3)),$$
since $R_0(\tilde{z})$ is a first order differential operator. Hence, solving the equation for $d$ by Proposition \ref{thm:dEq} we  obtain a unique solution $d$ which belongs to its
optimal regularity class. This completes the proof for the case $\Omega=\RR^3$.

\vspace{.1cm}
\noindent
\emph{Step 2:} Let $\Omega=\RR^3_+$ and consider the case of constant coefficients, i.e. we linearize at a constant vector $\tilde{z}=(\tilde{u},\tilde{d})\in\RR^3\times\RR^3$ such
that $|\tilde{d}|_2=1$. We first reduce \eqref{eq:princLin} to the case $(u_0,d_0,g)=(0,0,0)$, by applying Proposition \ref{thm:dEq} for the case $\Omega=\RR^3_+$ and \cite[Theorem 7.2.1]{PrSi16} for the
case of no-slip boundary conditions.

In the half space $\RR_+^3$, we replace the spatial co-variable $\xi$ by $\xi-i\nu\partial_y$, where $y>0$ and $\xi\perp\nu$. As in Step 1, we extract the Schur complement of $u$.
To this end, we consider the differential operator
\begin{align*}
M_{{d}}(\lambda,\xi-i\nu\partial_y)d&=\gamma \lambda d+m_{\tilde{d}}(\xi-i\nu\partial_y)d =\gamma \lambda d-\mathcal{A}_{\tilde{d}}(i\xi+\nu\partial_y)d,
\end{align*}
supplemented with the homogeneous boundary condition
$$
\mathcal{B}_{\tilde{d}}(i\xi+\nu\partial_y)d=0\ \text{for $y=0$}.
$$
We claim that for $\real \lambda\ge 0$, $(\lambda,\xi)\neq (0,0)$ with $\xi\perp\nu$, the operator $M_d$ is invertible. Indeed, consider the equation
$$
M_{{d}}(\lambda,\xi-i\nu\partial_y)\eta=f
$$
for given $f$. In a first step, we extend $f$ from $\R_+$ to a function $\tilde{f}$ on $\R$ and solve the full space problem
$$
\gamma \lambda \tilde{\eta}-\mathcal{A}_{\tilde{d}}(i\xi+\nu\partial_y)\tilde{\eta}=\tilde{f},\quad y\in\RR,
$$
by the classical Mihklin multiplier theorem with respect to  the variable $y$ and with the help of \eqref{eq:normell}. This yields a unique solution $\tilde{\eta}$.
Next, we consider the boundary value problem
\begin{equation}
\left\{\begin{aligned}
\gamma \lambda \hat{\eta}-\mathcal{A}_{\tilde{d}}(i\xi+\nu\partial_y)\hat{\eta}&=0&&y>0 ,\\
\mathcal{B}_{\tilde{d}}(i\xi+\nu\partial_y)\hat{\eta} &=h &&y=0,
\end{aligned}\right.
\end{equation}
with $h:=-\mathcal{B}_{\tilde{d}}(i\xi+\nu\partial_y)\tilde{\eta}$, to obtain a unique solution $\hat{\eta}$ by Remark \ref{rem:LS}. This shows that $M_d$ is invertible, which in turn implies that
$$d=M_{{d}}(\lambda,\xi-i\nu\partial_y)^{-1}(f_d+iR_0(\xi-i\nu\partial_y)u).$$
Therefore, the Schur complement of $u$ is given by
\begin{equation*}
M_u(\lambda,\xi-i\nu\partial_y)
-\lambda R_1(\xi-i\nu\partial_y)^{\sf T}M_{{d}}(\lambda,\xi-i\nu\partial_y)^{-1}R_0(\xi-i\nu\partial_y).
\end{equation*}

We will now show that the \emph{Lopatinskii-Shapiro condition} is satisfied. To be precise, this means that for all $\real \lambda\ge 0$, $(\lambda,\xi)\neq (0,0)$ with $\xi\perp\nu$, the problem
\begin{equation}\label{eq:LS2}
\left\{\begin{aligned}
M_u(\lambda,\xi-i\nu\partial_y)u+i\lambda R_1(\xi-i\nu\partial_y)^{\sf T}d&=0&&y>0 ,\\
-iR_0(\xi-i\nu\partial_y)u+M_{{d}}(\lambda,\xi-i\nu\partial_y)d&=0&&y>0,\\
u &=0 &&y=0,\\
\mathcal{B}_{\tilde{d}}(i\xi+\nu\partial_y)d &=0 &&y=0,
\end{aligned}\right.
\end{equation}
admits only the trivial solution $u=d=0$ in $L_2(\RR_+)$.
Let us split $d=d_1+d_2$, where $d_1=P_{\tilde{d}}d$ and $d_2=(I-P_{\tilde{d}})d$. Since
$$(I-P_{\tilde{d}})R_0(\xi-i\nu\partial_y)u=0,$$
by the definition of $R_0$, we may conclude from equation $\eqref{eq:LS2}_2$ that $d_2=0$. Indeed, this can be seen as in the proof of Proposition \ref{pro:LopShap_d}. Therefore, we may replace $d$ by $d_1=P_{\tilde{d}}d$ in \eqref{eq:LS2}. As in step 1, we will test $\eqref{eq:LS2}_1$ with $\bar{u}$ and $\eqref{eq:LS2}_2$ with $\bar{\lambda}\bar{d}_1$ and integrate by parts with respect to the
variable $y>0$. Assumption (P) then yields the estimate
\begin{equation*}
0\ge \real \lambda[\|u\|_{L_2(\RR_+)}^2+(m_{\tilde{d}}(\xi-i\nu\partial_y)d_1|d_1)_{L_2(\RR_+)}]\\
+|\lambda|^2\|d_1\|_{L_2(\RR_+)}^2+|\xi|_2^2\|u\|_{L_2(\RR_+)}^2+\|\partial_y u\|_{L_2(\RR_+)}^2.
\end{equation*}
Since, by Proposition \ref{pro:LopShap_d},
$$(m_{\tilde{d}}(\xi-i\nu\partial_y)d_1|d_1)_{L_2(\RR_+)}=-(\mathcal{A}_{\tilde{d}}(i\xi+\nu\partial_y)d_1|d_1)_{L_2(\RR_+)}\ge \alpha(|\xi|^2\|d_1\|_{L_2(\RR_+)}^2+\|\partial_yd_1\|_{L_2(\RR_+)}^2),$$
we may conclude that $\partial_y u=0$, hence $u=0$ as $u\in L_2(\RR_+)$. Inserting this information into $\eqref{eq:LS2}_2$, the boundary condition $\eqref{eq:LS2}_4$ and Proposition \ref{pro:LopShap_d} imply $d_1=0$. This shows that the Lopatinskii-Shapiro condition is satisfied.

We may now employ half-space theory for $u$ by the methods in \cite[Section 6]{BoPr07} or \cite[Section 2]{Pr18} or \cite[Section 7.2]{PrSi16} to prove maximal $L_p$-regularity for the half space $\RR_+^3$.
Having a unique solution $u$ in its optimal regularity class, it follows that
$$
R_0(\tilde{z})u\in L_{p,\mu}((0,T);H_p^1(\RR_+^3;\RR^3)).
$$
Hence, we solve the equation for $d$ by Proposition \ref{thm:dEq} to obtain a unique solution $d$ which belongs to its optimal regularity class. This completes the proof for the case $\Omega=\RR_+^3$.

\vspace{.1cm}
\noindent
\emph{Step 3:} The results of Steps 1 and  2 extend by perturbation arguments to a bent half-space and to the case of variable coefficients with small deviation from constant ones. We
then may apply a localization procedure to cover the case of general domains with boundary of class $C^3$ and variable coefficients. For details we refer at this point e.g. to
\cite[Sections 6.3 \& 7.3]{PrSi16}. This completes the proof of Theorem \ref{thm:MaxRegLin}.
\end{proof}

We will now  rewrite \eqref{eq:iEL2} in a more abstract form. To this end, let
$$
X_0=L_{p,\sigma}(\Omega;\RR^3)\times H_p^1(\Omega;\RR^3),
$$
where $L_{p,\sigma}(\Omega;\RR^3):=\mathbb{P}_HL_p(\Omega;\RR^3)$ and $\mathbb{P}_H$ denotes the Helmholtz projection. Furthermore, let
$$
X_1:=\{u\in H_{p,\sigma}^2(\Omega;\RR^3)\mid u=0\ \text{on}\ \partial\Omega\}\times H_p^3(\Omega;\RR^3),
$$
where $H_{p,\sigma}^2(\Omega;\RR^3):=H_{p}^2(\Omega;\RR^3)\cap L_{p,\sigma}(\Omega;\RR^3)$ and let
$$
X_{\gamma,\mu}:=(X_0,X_1)_{\mu-1/p,p}
$$
be the space of the initial data. In fact, it holds that
\begin{equation}\label{eq:X_gamma_mu}
X_{\gamma,\mu}=\{u\in W_{p,\sigma}^{2\mu-2/p}(\Omega;\RR^3)\mid u=0\ \text{on}\ \partial\Omega\}\times W_{p}^{2\mu+1-2/p}(\Omega;\RR^3).
\end{equation}

Observe next that
$$
X_{\gamma,\mu}\hookrightarrow C^{1}(\overline{\Omega};\RR^3)\times C^{2}(\overline{\Omega};\RR^3),
$$
by our assumption \eqref{eq:Assum_for_p} on $p$ and $\mu$.  Given any $\tilde{z}=(\tilde{u},\tilde{d})\in X_{\gamma,\mu}$, we define
\begin{equation}\label{eq:Def_sA}
\sA(\tilde{z}):=
\begin{pmatrix}
\frac{1}{\rho}\mathbb{P}_HA_u(\tilde{z})+\frac{1}{\rho\gamma(\tilde{z})}\mathbb{P}_HR_1(\tilde{z})^{\sf T}R_0(\tilde{z}) & -\frac{1}{\rho\gamma(\tilde{z})}\mathbb{P}_HR_1(\tilde{z})^{\sf T}A_d(\tilde{z})\\
-\frac{1}{\gamma(\tilde{z})}R_0(\tilde{z}) & \frac{1}{\gamma(\tilde{z})}A_d(\tilde{z})
\end{pmatrix}
\end{equation}
and, for sufficiently smooth $z=(u,d)$, we introduce, following \eqref{eq:BC2}, the boundary operator $\sB(\tilde{z})$ by
\begin{align}\label{eq:Def_sB}
\begin{split}
\sB(\tilde{z})z:=\cB_{\tilde{d}}(\nabla) d &=2k_3\nabla d\cdot \nu +2P_{\tilde{d}}(k_1\diver d\cdot I-k_3(\nabla d)^{\sf T})\cdot \nu\\
&+2(k_2-k_3)(\tilde{d}\cdot \curl d)(\tilde{d}\times\nu) +2(k_2+k_4)P_{\tilde{d}}((\nabla d)^{\sf T}-\diver d\cdot I)\cdot \nu,
\end{split}
\end{align}
Finally, let
\begin{equation}\label{eq:Def_sF}
\sF(z):=
\begin{pmatrix}
\frac{1}{\rho}\mathbb{P}_HF_u(z)-\frac{1}{\rho\gamma(\tilde{z})}\mathbb{P}_HR_1^{\sf T}(z)F_d(z)\\
\frac{1}{\gamma(\tilde{z})}F_d(z)
\end{pmatrix}.
\end{equation}
We note that the definition of $\sA(\tilde{z})$ as well as of $\sF(z)$ results from \eqref{eq:iEL2} by applying the Helmholtz projection to the first equation and  by substituting $\partial_t d$ from the second equation into the first equation of \eqref{eq:iEL2}.

With these definitions, system \eqref{eq:iEL2} can be rewritten as
\begin{equation}
\label{eq:iEL3}
\left\{\begin{aligned}
\partial_t z+\sA(z)z &=\sF(z)&&\text{in}&&(0,T)\times\Omega ,\\
\sB(z)z &=0&&\text{in}&&(0,T)\times\partial\Omega ,\\
z(0)&=z_0 &&\text{in}&&\{0\}\times\Omega.
\end{aligned}\right.
\end{equation}
For the sake of readability, for $0<T<\infty$, we further define $J_T=[0,T]$,
$$\mathbb{E}_{0,\mu}(J_T):=L_{p,\mu}(J_T;X_0),\quad \mathbb{E}_{1,\mu}(J_T):=H_{p,\mu}^1(J_T;X_0)\cap L_{p,\mu}(J_T;X_1)$$
and
$$\mathbb{F}_{\mu}(J_T):=W_{p,\mu}^{1-1/2p}(J_T;L_p(\partial\Omega;\RR^3))\cap L_{p,\mu}(J_T;W_{p}^{2-1/p}(\partial\Omega;\RR^3)).$$
Moreover, we set
\begin{equation}\label{eq:hatTrace}
\hat{X}_{\gamma,\mu}:=\{z=(u,d)\in X_{\gamma,\mu}\mid |d(x)|_2=1,\ x\in\Omega\}
\end{equation}
and, for $\tilde{z}\in X_{\gamma,\mu}$,
$$\mathbb{D}_\mu(\tilde{z},T):=\{(f,g,z_0)\in \mathbb{E}_{0,\mu}(J_T)\times \mathbb{F}_{\mu}(J_T)\times X_{\gamma,\mu}\mid \sB(\tilde{z})z_0=\tr_{t=0}g\}.$$
Then the  following result for the linearized system
\begin{equation}
\label{eq:princLin2}
\left\{\begin{aligned}
\partial_t z+\sA(\tilde{z})z &=f&&\text{in}&&(0,T)\times\Omega ,\\
\sB(\tilde{z})z &=g&&\text{in}&&(0,T)\times\partial\Omega ,\\
z(0)&=z_0 &&\text{in}&&\{0\}\times\Omega,
\end{aligned}\right.
\end{equation}
is a direct consequence of Theorem \ref{thm:MaxRegLin}.

\begin{corollary}\label{cor:MaxRegLin2}
Let the assumptions of Theorem \ref{thm:MaxRegLin} be satisfied. Then, for any $\tilde{z}\in \hat{X}_{\gamma,\mu}$ and all $(f,g,z_0)\in\mathbb{D}_{\mu}(\tilde{z},T)$, the linear problem \eqref{eq:princLin2} admits a unique solution $z\in\mathbb{E}_{1,\mu}(J_T)$.
\end{corollary}

In the situation of Corollary \ref{cor:MaxRegLin2}, for any $\tilde{z}\in \hat{X}_{\gamma,\mu}$, the mapping
$$L(\tilde{z})=(\partial_t+\sA(\tilde{z}),\sB(\tilde{z}),\tr_{t=0}):\mathbb{E}_{1,\mu}(J_T)\to \mathbb{D}_{\mu}(\tilde{z},T)$$
is linear, bounded and invertible. Denoting by
\begin{equation}\label{eq:SolOper}
S(\tilde{z})=L(\tilde{z})^{-1}
\end{equation}
the inverse operator, the open mapping theorem implies that
$S(\tilde{z}):\mathbb{D}_{\mu}(\tilde{z},T)\to \mathbb{E}_{1,\mu}(J_T)$
is bounded. Therefore, there exists a constant $C=C(T)>0$ such that the unique solution $z$ of \eqref{eq:princLin2} satisfies
\begin{equation}\label{eq:MRest}
\|z\|_{\mathbb{E}_{1,\mu}(J_T)}\le C\left(\|f\|_{\mathbb{E}_{0,\mu}(J_T)}+
\|g\|_{\mathbb{F}_{\mu}(J_T)}+\|z_0\|_{X_{\gamma,\mu}}\right).
\end{equation}
With the help of extension-restriction arguments one can prove that in case $z_0=0$, the constant $C=C(T)$ is uniform in $T\in (0,T_*]$ for some given and
fixed $T_*\in (0,\infty)$, see e.g. \cite[Proposition 4.1 (b)]{DuShaSi24}.

We remark that by assumptions (R) and \eqref{eq:Assum_for_p}, the mapping
$$X_{\gamma,\mu}\ni \tilde{z}\mapsto L(\tilde{z})\in \cL({_0}\mathbb{E}_{1,\nu}(J_T),\mathbb{E}_{0,\nu}(J_T)\times {_0}\mathbb{F}_{\nu}(J_T))$$
is continuous, where the lower left subscript 0 means that the trace at $t=0$ vanishes. Indeed, for $\sA$ continuity follows from a direct calculation and concerning $\sB$, one may use the fact that the space $\mathbb{F}_\mu(J_T)$ is a Banach algebra, cf. also \cite[Lemma B.1]{DuShaSi24}. Corollary \ref{cor:MaxRegLin2} then implies that the mapping
$$\hat{X}_{\gamma,\mu}\ni\tilde{z}\mapsto S(\tilde{z})\in \cL(\mathbb{E}_{0,\nu}(J_T)\times {_0}\mathbb{F}_{\nu}(J_T),{_0}\mathbb{E}_{1,\nu}(J_T))$$
is continuous, since $\hat{X}_{\gamma,\mu}$ is a subset of $X_{\gamma,\mu}$, see \eqref{eq:hatTrace}, and since inversion is smooth. Here the solution operator $S$ is defined in \eqref{eq:SolOper}, but restricted to trivial initial data.

Next, let us define nonlinear mappings $(\cA,\cB)$ by
$$
\cA(z)=\sA(z)z\quad\text{and}\quad\cB(z)=\sB(z)z,
$$
where  $(\sA,\sB)$ are given in  \eqref{eq:Def_sA} and \eqref{eq:Def_sB}. Let further $\sF$ be given as in  \eqref{eq:Def_sF}. Then the
functions $\cA$, $\cB$ and $\sF$ enjoy the following regularity properties.

\begin{lemma}\label{pro:Nonlin_Est}
Let  $1<p<\infty$, $\mu \in (\frac12 + \frac{5}{2p},1]$ and assume (P), (R). Then
$$
\cA,\sF\in C^1(\mathbb{E}_{1,\mu}(J_T);\mathbb{E}_{0,\mu}(J_T)),\ \cB\in C^1(\mathbb{E}_{1,\mu}(J_T);\mathbb{F}_{\mu}(J_T)),
$$
with
$$
\cA'(z_*)z=\sA(z_*)z+[\sA'(z_*)z]z_*,\ \cB'(z_*)z=\sB(z_*)z+[\sB'(z_*)z]z_*,
$$
where $z,z_*\in \mathbb{E}_{1,\mu}(J_T)$.

Moreover, given $T_0,M>0$, then for any $T\in (0,T_0]$ and any $z_*=(u_*,d_*)\in \mathbb{E}_{1,\mu}(J_T)$, $z\in \mathbb{E}_{1,\mu}(J_T)$ with $z(0)=0$ satisfying
$$\|\tr_{\partial\Omega}\nabla^jd_*\|_{\mathbb{F}_{\mu}(J_T)},\ \|z_*\|_{C(J_T;X_{\gamma,\mu})},\ \|z_*\|_{\mathbb{E}_{1,\mu}(J_T)},\ \|z\|_{\mathbb{E}_{1,\mu}(J_T)}\le M,$$
for $j\in\{0,1\}$, the estimate
$$\|\mathcal{H}(z_*+z)-\mathcal{H}(z_*)-\mathcal{H}'(z_*)z\|_{\mathbb{X}}\le \varepsilon(\|z\|_{\mathbb{E}_{1,\mu}(J_T)})\|z\|_{\mathbb{E}_{1,\mu}(J_T)}$$
holds for $(\mathcal{H},\mathbb{X})\in\{(\cA,\mathbb{E}_{0,\mu}(J_T)),(\sF,\mathbb{E}_{0,\mu}(J_T)),(\cB,\mathbb{F}_{\mu}(J_T))\}$. Here $\varepsilon:\RR_+\to\RR_+$ is continuous such that $\varepsilon(r)\to 0$ as $r\to 0$.

If in addition $\bar{z}=(\bar{u},\bar{d})\in \mathbb{E}_{1,\mu}(J_T)$ with $\bar{z}(0)=z_*(0)$ satisfies
$$\|\tr_{\partial\Omega}\nabla^j\bar{d}\|_{\mathbb{F}_{\mu}(J_T)},\ \|\bar{z}\|_{C(J_T;X_{\gamma,\mu})},\ \|\bar{z}\|_{\mathbb{E}_{1,\mu}(J_T)}\le M,$$
for $j\in\{0,1\}$, then the following estimate holds
$$\|\mathcal{H}'(z_*)z-\mathcal{H}'(\bar{z})z\|_{\mathbb{X}}\le \varepsilon(\|z_*-\bar{z}\|_{\mathbb{E}_{1,\mu}(J_T)})\|z\|_{\mathbb{E}_{1,\mu}(J_T)},$$
where the choice of $(\mathcal{H},\mathbb{X})$ is as above.
\end{lemma}

\begin{proof}
For the proof one may follow the strategy of \cite[Proof of Proposition B.3 and Lemma B.2]{DuShaSi24}. Indeed, by \eqref{eq:Assum_for_p}, the space $\mathbb{F}_\mu(J_T)$ is a
Banach algebra and the nonlinearities in $\mathcal{B}$ are of the form
$$
\left(d\otimes d\right)\cdot \phi(\nabla d)\quad\text{or}\quad (d|\curl d)d,$$
where
$$
\phi(\nabla d)=(\nabla d)^{\sf T}\quad\text{or}\quad \phi(\nabla d)=\diver d\cdot I=\Tr(\nabla d)\cdot I.$$
\end{proof}

For the proof of local well-posedness in the next section, we consider the following nonautonomous problem
\begin{equation}
\label{eq:nonaut1}
\left\{\begin{aligned}
\partial_t z+\cA'({z}_*(t))z -\sF'(z_*(t))z&=f&&\text{in}&&(0,T)\times\Omega ,\\
\cB'({z}_*(t))z &=g&&\text{in}&&(0,T)\times\partial\Omega ,\\
z(0)&=z_0 &&\text{in}&&\{0\}\times\Omega.
\end{aligned}\right.
\end{equation}
Here,
$$
\cA(z):=\sA(z)z\quad\text{and}\quad\cB(z):=\sB(z)z,
$$
and $z_*\in \mathbb{E}_{1,\mu}(J_T)$ is a given function. By Lemma  \ref{pro:Nonlin_Est}, $\cA,\cB$ and $\sF$ are continuously differentiable,
with
$$
\cA'(z_*)z=\sA(z_*)z+[\sA'(z_*)z]z_*,\ \cB'(z_*)z=\sB(z_*)z+[\sB'(z_*)z]z_*.
$$
\begin{proposition}\label{pro:LinAuxPrb}
Suppose that the conditions of Theorem \ref{thm:MaxRegLin} are satisfied.  Then, for any $z_*\in \mathbb{E}_{1,\mu}(J_T)$ with $z_*(t)\in \hat{X}_{\gamma,\mu}$, $t\in J_T$, and all
$$(f,g,z_0)\in \mathbb{E}_{0,\mu}(J_T)\times \mathbb{F}_{\mu}(J_T)\times X_{\gamma,\mu}$$
such that $\mathcal{B}'(z_*(0))=\tr_{t=0}g$, the linear problem \eqref{eq:nonaut1} admits a unique solution $z\in\mathbb{E}_{1,\mu}(J_T)$ satisfying the estimate \eqref{eq:MRest}.

Moreover, in case $z_0=0$, the constant $C=C(T)$ in \eqref{eq:MRest} is uniform in $T\in (0,T_*]$ for some given and fixed $T_*\in (0,\infty)$.
\end{proposition}

\begin{proof}
In a first step, one proves the assertion for the nonautonomous problem
\begin{equation}
\label{eq:nonaut2}
\left\{\begin{aligned}
\partial_t z+\sA({z}_*(t))z&=f&&\text{in}&&(0,T)\times\Omega ,\\
\sB({z}_*(t))z &=g&&\text{in}&&(0,T)\times\partial\Omega ,\\
z(0)&=z_0 &&\text{in}&&\{0\}\times\Omega.
\end{aligned}\right.
\end{equation}
To this end, we follow  the strategy of the proof of Proposition 4.2 in \cite{DuShaSi24}, take into account that for any \emph{fixed} $s\in J_T$,
problem \eqref{eq:princLin2} with $\tilde{z}=z_*(s)\in \hat{X}_{\gamma,\mu}$ has the property of $L_p$-maximal regularity by Corollary \ref{cor:MaxRegLin2} and that $\mathbb{F}_\mu(J_T)$ is a Banach algebra by \eqref{eq:Assum_for_p}. The result then follows via a suitable decomposition of the interval $[0,T]$.

In a second step, we  include the terms $[\sA'(z_*)z]z_*$, $[\sB'(z_*)z]z_*$ and $\sF'(z_*)z$ by a perturbation argument as in the proof of Proposition 4.3 in \cite{DuShaSi24}. The details of this
procedure hence are omitted.
\end{proof}

\section{Proofs of the Main Results}
We start with the unique strong solvability of the abstract system \eqref{eq:iEL3}.

\begin{proposition}\label{thm:LWP}
Let  $1<p<\infty$, $\mu \in (\frac12 + \frac{5}{2p},1]$ and assume (F),(P) and (R). Then, for every $z_0\in \hat{X}_{\gamma,\mu}$ satisfying  $\sB(z_0)z_0=0$, there exists $T=T(z_0)>0$ such
that problem \eqref{eq:iEL3} has a unique solution $z\in \mathbb{E}_{1,\mu}(J_T)$.
\end{proposition}
\begin{proof}
We claim first that there exists $T_0>0$ and $z_*\in \mathbb{E}_{1,\mu}(J_{T_0})$ satisfying  $z_*(t)\in \hat{X}_{\gamma,\mu}$ for all $t\in J_{T_0}$ and such that $z_*(0)=z_0$. Indeed, for given
$$u_0\in \{u\in W_{p,\sigma}^{2\mu-2/p}(\Omega;\RR^3)\mid u=0\ \text{on}\ \partial\Omega\}$$
we obtain from \cite[Theorem 7.3.1]{PrSi16} a function
$$u_*\in H_{p,\mu}^1(\R_+;L_{p,\sigma}(\Omega;\RR^3))\cap L_{p,\mu}(\R_+;H_{p,\sigma}^2(\Omega;\RR^3))$$
such that $u_*(0)=u_0$. Concerning $d_0\in W_{p}^{2\mu+1-2/p}(\Omega;\RR^3)$, we first extend $d_0$ to some $\tilde{d}_0\in W_{p}^{2\mu+1-2/p}(\RR^3;\RR^3)$ and solve the full space problem
\begin{equation}
\left\{\begin{aligned}
\partial_t \tilde{d}+\omega\tilde{d}-\Delta\tilde{d}&=0&&\text{in}&&\RR_+\times\RR^3 ,\\
\tilde{d}(0)&=\tilde{d}_0 &&\text{in}&&\{0\}\times\RR^3,
\end{aligned}\right.
\end{equation}
for some $\omega>0$, to obtain a unique solution
$$
\tilde{d}\in H_{p,\mu}^1(\R_+;H_{p}^1(\RR^3;\RR^3))\cap L_{p,\mu}(\R_+;H_{p}^3(\RR^3;\RR^3)),
$$
by \cite[Theorem 6.1.11 \& Section 6.1.5 (i)]{PrSi16}. The restricted function $\hat{d}:=\tilde{d}|_{\Omega}$ then satisfies
$$
\hat{d}\in H_{p,\mu}^1(\R_+;H_{p}^1(\Omega;\RR^3))\cap L_{p,\mu}(\R_+;H_{p}^3(\Omega;\RR^3))
$$
and $\hat{d}(0)=d_0$. Since $z_0=(u_0,d_0)\in\hat{X}_{\gamma,\mu}$, we know that $|d_0|_2=1$.
Let us recall that
\begin{equation}\label{eq:emb_hat_d}
\hat{d}\in C(\R_+;C^2(\overline{\Omega};\RR^3))
\end{equation}
by our assumption \eqref{eq:Assum_for_p} on $p$ and $\mu$.  Therefore, by continuity and compactness, there exist $T_0>0$ and $\alpha,\beta>0$ such that
$$0<\alpha\le |\hat{d}(t,x)|_2\le\beta<\infty$$
for all $(t,x)\in [0,T_0]\times \overline{\Omega}$. This estimate and \eqref{eq:emb_hat_d} in turn imply that the function
$$d_*(t,x):=\frac{1}{|\hat{d}(t,x)|_2}\hat{d}(t,x)$$
satisfies
$$d_*\in H_{p,\mu}^1(J_{T_0};H_{p}^1(\Omega;\RR^3))\cap L_{p,\mu}(J_{T_0};H_{p}^3(\Omega;\RR^3))$$
with $|d_*(t,x)|_2=1$ for all $(t,x)\in J_{T_0}\times\overline{\Omega}$ and $d_*(0)=\frac{1}{|d_0|_2}d_0=d_0$. Hence, $z_*:=(u_*,d_*)$ has the desired properties.

For the remainder of the proof, we define
$$
\sA_*(t)z:=\cA'({z}_*(t))z -\sF'(z_*(t)),\quad \sB_*(t)z:=\cB'({z}_*(t))z
$$
and solve in a first step the linear problem
\begin{equation}
\label{eq:LWP1}
\left\{\begin{aligned}
\partial_t z+\sA_*(t)z&=\cA'(z_*)z_*-\cA(z_*)+\sF(z_*)-\sF'(z_*)z_*&&\text{in}&&(0,T_0)\times\Omega ,\\
\sB_*(t)z &=\cB'(z_*)z_*-\cB(z_*)&&\text{in}&&(0,T_0)\times\partial\Omega ,\\
z(0)&=z_0 &&\text{in}&&\{0\}\times\Omega,
\end{aligned}\right.
\end{equation}
by Proposition \ref{pro:LinAuxPrb} to obtain a unique solution $w\in\mathbb{E}_{1,\mu}(J_{T_0})$ of \eqref{eq:LWP1}. We note that the compatibility condition
$$\sB_*(0)z_0 =\cB'(z_0)z_0-\cB(z_0)$$
at $t=0$ is satisfied, since $\cB(z_0)=\sB(z_0)z_0$ by definition of $\cB$ and $\sB(z_0)z_0=0$ by assumption.

Given any $R_0>0$, we define for $(T,R)\in (0,T_0]\times (0,R_0]$ the set
$$\Sigma(T,R):=\{z\in \mathbb{E}_{1,\mu}(J_{T})\mid\|z-w\|_{\mathbb{E}_{1,\mu}(J_{T})}\le R,\ \tr_{t=0}z=z_0\},$$
which is a closed subset of  $\mathbb{E}_{1,\mu}(J_{T})$.
For a given $\hat{z}\in \Sigma(T,R)$, we solve the linear problem
\begin{equation}
\label{eq:LWP2}
\left\{\begin{aligned}
\partial_t z+\sA_*(t)z&=\cA'(z_*)\hat{z}-\cA(\hat{z})+\sF(\hat{z})-\sF'(z_*)\hat{z}&&\text{in}&&(0,T)\times\Omega ,\\
\sB_*(t)z &=\cB'(z_*)\hat{z}-\cB(\hat{z})&&\text{in}&&(0,T)\times\partial\Omega ,\\
z(0)&=z_0 &&\text{in}&&\{0\}\times\Omega,
\end{aligned}\right.
\end{equation}
to obtain a unique solution $z=\cT(\hat{z})\in \mathbb{E}_{1,\mu}(J_{T})$ by Proposition \ref{pro:LinAuxPrb}. We note that the compatibility condition $\sB_*(0)z_0 =\cB'(z_0)z_0-\cB(z_0)$ at $t=0$ is again satisfied.

We see  that $z\in \Sigma(T,R)$ solves \eqref{eq:iEL3} if and only if $z$ is a fixed point of $\cT$, i.e. $\cT(z)=z$. For the latter  purpose, we will employ the contraction principle.
It follows from \eqref{eq:MRest}, \eqref{eq:LWP1} and \eqref{eq:LWP2} that there exists a constant $C>0$, not depending on $T\in (0,T_0]$, such that
\begin{align*}
\|\cT(\hat{z})-w\|_{\mathbb{E}_{1,\mu}(J_{T})}\le &C(\|\cA(\hat{z})-\cA(z_*)-\cA'(z_*)(\hat{z}-z_*)\|_{\mathbb{E}_{0,\mu}(J_{T})}
+\|\sF(\hat{z})-\sF(z_*)-\sF'(z_*)(\hat{z}-z_*)\|_{\mathbb{E}_{0,\mu}(J_{T})}\\
&\quad +\|\cB(\hat{z})-\cB(z_*)-\cB'(z_*)(\hat{z}-z_*)\|_{\mathbb{F}_{\mu}(J_{T})}),
\end{align*}
since $\tr_{t=0}(\cT(\hat{z})-w)=0$.
Observing that
$$
\|\hat{z}-z_*\|_{\mathbb{E}_{1,\mu}(J_{T})}\le \|\hat{z}-w\|_{\mathbb{E}_{1,\mu}(J_{T})}+\|w-z_*\|_{\mathbb{E}_{1,\mu}(J_{T})}\le R+\|w-z_*\|_{\mathbb{E}_{1,\mu}(J_{T})},
$$
Lemma \ref{pro:Nonlin_Est} yields the estimate
\begin{align*}
\|\cT(\hat{z})-w\|_{\mathbb{E}_{1,\mu}(J_{T})}&\le\varepsilon(\|\hat{z}-z_*\|_{\mathbb{E}_{1,\mu}(J_{T})})\|\hat{z}-z_*\|_{\mathbb{E}_{1,\mu}(J_{T})}\\
&\le \varepsilon(R+\|w-z_*\|_{\mathbb{E}_{1,\mu}(J_{T})})(R+\|w-z_*\|_{\mathbb{E}_{1,\mu}(J_{T})}) \le R,
\end{align*}
provided $T,R$ are chosen sufficiently small. Here we used the fact that $w,z_*\in \mathbb{E}_{1,\mu}(J_{T_0})$ are fixed functions and
$$\|w-z_*\|_{\mathbb{E}_{1,\mu}(J_{T})}\to 0$$
as $T\to 0$. Furthermore, for $\hat{z},\bar{z}\in\Sigma(T,R)$, we obtain
\begin{align*}
\|\cT(\hat{z})-\cT(\bar{z})\|_{\mathbb{E}_{1,\mu}(J_{T})}\le &C (\|\cA(\hat{z})-\cA(\bar{z})-\cA'(z_*)(\hat{z}-\bar{z})\|_{\mathbb{E}_{0,\mu}(J_{T})}\\
&+\|\sF(\hat{z})-\sF(\bar{z})-\sF'(z_*)(\hat{z}-\bar{z})\|_{\mathbb{E}_{0,\mu}(J_{T})}\\
&+\|\cB(\hat{z})-\cB(\bar{z})-\cB'(z_*)(\hat{z}-\bar{z})\|_{\mathbb{F}_{\mu}(J_{T})}).
\end{align*}
Estimating
\begin{align*}
\|\cA(\hat{z})-\cA(\bar{z})-\cA'(z_*)(\hat{z}-\bar{z})\|_{\mathbb{E}_{0,\mu}(J_{T})} &\le
\|\cA(\hat{z})-\cA(\bar{z})-\cA'(\hat{z})(\hat{z}-\bar{z})\|_{\mathbb{E}_{0,\mu}(J_{T})}\\
& \quad +\|(\cA'(z_*)-\cA'(\hat{z}))(\hat{z}-\bar{z})\|_{\mathbb{E}_{0,\mu}(J_{T})}
\end{align*}
and applying again Lemma \ref{pro:Nonlin_Est} we verify that
$$
\|\cA(\hat{z})-\cA(\bar{z})-\cA'(\hat{z})(\hat{z}-\bar{z})\|_{\mathbb{E}_{0,\mu}(J_{T})}
\le \varepsilon(\|\hat{z}-\bar{z}\|_{\mathbb{E}_{1,\mu}(J_{T})})\|\hat{z}-\bar{z}\|_{\mathbb{E}_{1,\mu}(J_{T})}
$$
and
$$
\|(\cA'(z_*)-\cA'(\hat{z}))(\hat{z}-\bar{z})\|_{\mathbb{E}_{0,\mu}(J_{T})}\le
\varepsilon(\|z_*-\hat{z}\|_{\mathbb{E}_{1,\mu}(J_T)})\|\hat{z}-\bar{z}\|_{\mathbb{E}_{1,\mu}(J_T)}.
$$
Here we used the fact that there exists a constant $M>0$ such that for all $(T,R)\in (0,T_0]\times (0,R_0]$ and for every $\hat{z}=(\hat{u},\hat{d})\in \Sigma(T,R)$ we have
$$
\|\tr_{\partial\Omega}\nabla^j\hat{d}\|_{\mathbb{F}_{\mu}(J_T)},\ \|\hat{z}\|_{C(J_T;X_{\gamma,\mu})},\ \|\hat{z}\|_{\mathbb{E}_{1,\mu}(J_T)}\le M,\quad j\in\{0,1\}.
$$
We further see that
$$\|\hat{z}-\bar{z}\|_{\mathbb{E}_{1,\mu}(J_{T})}\le \|\hat{z}-w\|_{\mathbb{E}_{1,\mu}(J_{T})}+\|w-\bar{z}\|_{\mathbb{E}_{1,\mu}(J_{T})}\le 2R$$
and
\begin{align*}
\|z_*-\hat{z}\|_{\mathbb{E}_{1,\mu}(J_T)}\le \|z_*-w\|_{\mathbb{E}_{1,\mu}(J_T)}+\|w-\hat{z}\|_{\mathbb{E}_{1,\mu}(J_T)}
\le \|z_*-w\|_{\mathbb{E}_{1,\mu}(J_T)}+R .
\end{align*}
Analogous estimates for the terms involving $\sF$ and $\cB$ finally imply
$$
\|\cT(\hat{z})-\cT(\bar{z})\|_{\mathbb{E}_{1,\mu}(J_{T})}\le\frac{1}{2}\|\hat{z}-\bar{z}\|_{\mathbb{E}_{1,\mu}(J_T)},
$$
provided $T$ and $R$ are chosen sufficiently small. Hence, $\cT:\Sigma(T,R)\to\Sigma(T,R)$ is a contraction and we obtain the existence and uniqueness of a fixed point
$z\in\Sigma(T,R)$ of $\cT$, which is then the unique solution of \eqref{eq:iEL3}.
\end{proof}

It follows from the definition of the Helmholtz projection $\mathbb{P}_H$ that any solution to \eqref{eq:iEL3} is a solution of \eqref{eq:elgenincom} or \eqref{eq:iEL2}. Indeed,
given $v\in L_p(\Omega;\RR^3)$, we have  $\mathbb{P}_Hv=v-\nabla\pi$, where $\pi\in \dot{H}_p^1(\Omega)$ solves the weak Neumann problem
$$(\nabla\pi|\nabla\phi)_{L_2(\Omega)}=(v|\nabla\phi)_{L_2(\Omega)},\quad \phi\in \dot{H}_{p'}^1(\Omega),$$
with $p'=p/(p-1)$. Hence, setting $v:=-\rho u\cdot\nabla u+\diver S$, it follows that
$$
\rho \partial_t u-v+\nabla\pi=\rho\partial_t u-\mathbb{P}_H v=0.
$$
Vice versa, any solution of \eqref{eq:elgenincom} or \eqref{eq:iEL2} is also a solution of \eqref{eq:iEL3}. This follows by applying  the Helmholtz projection to the first equation of
\eqref{eq:elgenincom} or \eqref{eq:iEL2}.


\begin{proof}[Proof of Theorem \ref{thm:LWP2}]
Note first that existence and regularity of a unique solution for \eqref{eq:elgenincom}  follows directly from Proposition \ref{thm:LWP}.

Concerning the property $|d(t,x)|_2=1$, we recall the differential equation for $d$
\begin{equation}\label{eq:d_LWP}
\gamma\partial_td+\gamma u\cdot \nabla d = P_d({\rm div}(\rho\frac{\partial\psi}{\partial\nabla d})-\rho\nabla_d\psi) {+\mu_V Vd + \mu_D P_d Dd}.
\end{equation}
Let us further recall from Section \ref{sec:compEop} that
\begin{align*}
  P_d\diver\left(\frac{\partial{\psi}}{\partial(\nabla d)}\right)&=2k_3(\Delta d+|\nabla d|_2^2d)+2(k_1-k_3)P_d\nabla\diver d + 2(k_2-k_3)\left((d\times\nabla)\otimes \curl d\right)\cdot  d\\
  &\quad +2(k_2-k_3)\left((d\times\nabla)\otimes d\right)\cdot \curl d
  -2(k_2-k_3)(d\cdot \curl d)P_d\curl d.
\end{align*}
Our goal is to derive an initial boundary value problem for the function $\varphi:=|d|_2^2-1$ by testing \eqref{eq:d_LWP} and the boundary condition for $d$ with $d$ and to show that $\varphi = 0$.
To this end, we observe that
$(\Delta d|d)=\frac{1}{2}\Delta\varphi-|\nabla d|_2^2$ and hence
$$
(\Delta d+|\nabla d|_2^2d|d)=\frac{1}{2}\Delta\varphi+|\nabla d|_2^2\varphi.
$$
Furthermore,
$(P_d\nabla\diver d|d)=(\nabla\diver d|P_dd)=-(\nabla\diver d|d)\varphi$ and
$$(d\cdot \curl d)(P_d\curl d|d)=(d\cdot \curl d)(\curl d|P_dd)=-(d\cdot \curl d)(\curl d|d)\varphi,$$
since $P_d^{\sf T}=P_d$ and $P_dd=-\varphi d$. Therefore, by \eqref{eq:perp_d}, we obtain
\begin{align*}
  \Big(P_d\diver\Big(\frac{\partial{\psi}}{\partial(\nabla d)}\Big)\Big| d\Big)&=k_3\Delta \varphi+2k_3|\nabla d|_2^2\varphi-2(k_1-k_3)(\nabla\diver d|d)\varphi\\
  &\quad + 2(k_2-k_3)(d\cdot \curl d)(\curl d|d)\varphi.
\end{align*}
For the remaining terms on the right side in \eqref{eq:d_LWP} we have
$$
(P_d\nabla_d\psi|d)=-(\nabla_d\psi|d)\varphi,\quad (P_dDd|d)=-(Dd|d)\varphi \mbox{ and }(Vd|d)=0,
$$
since $V$ is skew-symmetric. The terms on the left side of \eqref{eq:d_LWP} are treated as
$$
(\partial_td|d)=\frac{1}{2}\partial_t\varphi\quad\text{and}\quad (u\cdot\nabla d|d)=\frac{1}{2}(u|\nabla\varphi).
$$
Finally, we consider the contributions by  the  boundary condition for $\varphi$ by $d$. Let us recall from \eqref{eq:BC1} that
\begin{align*}
0&=P_d\frac{\partial\psi}{\partial(\nabla d)}\cdot\nu
=k_3\nabla d\cdot \nu+P_{{d}}(k_1\diver d\cdot I-k_3(\nabla d)^{\sf T})\cdot \nu\\
& \quad +(k_2-k_3)({d}\cdot \curl d)({d}\times\nu) +(k_2+k_4)P_{{d}}((\nabla d)^{\sf T}-\diver d\cdot I)\cdot \nu,
\end{align*}
which, multiplying  by $d$, yields
\begin{align*}
0=\frac{1}{2} k_3\partial_\nu\varphi-(k_1\diver d\cdot \nu-k_3(\nabla d)^{\sf T}\cdot \nu|d)\varphi
-(k_2+k_4)((\nabla d)^{\sf T}\cdot\nu-\diver d\cdot \nu|d)\varphi.
\end{align*}
Here we used again the identity $P_d d=-\varphi d$. Thus, we obtain the following problem for $\varphi$:
\begin{equation}
\label{eq:Prbl_varphi}
\left\{\begin{aligned}
\gamma\partial_t \varphi&=2k_3\Delta\varphi+G_1\varphi&&\text{in}&&\Omega ,\\
k_3\partial_\nu\varphi&=G_0\varphi
&&\text{on}&&\partial\Omega ,\\
\varphi(0)&=0 &&\text{in}&&\Omega,
\end{aligned}\right.
\end{equation}
where
\begin{align*}
G_1\varphi&:=2k_3|\nabla d|_2^2\varphi+\rho(\nabla_d\psi|d)\varphi-\mu_D(Dd|d)\varphi-\gamma(u|\nabla\varphi)\\
&\quad -2(k_1-k_3)(\nabla\diver d|d)\varphi +2(k_2-k_3)(d\cdot \curl d)(\curl d|d)\varphi,
\end{align*}
and
$$
G_0\varphi:=(k_2+k_4-k_3)((\nabla d)^{\sf T}\cdot\nu|d)\varphi-(k_2+k_4-k_1)\diver d (\nu|d)\varphi.
$$
Our aim is now to show  that $\varphi=0$ is the unique solution to \eqref{eq:Prbl_varphi}. To this end, we observe that if
$$d\in H_{p,\mu}^1(J_T;H_p^1(\Omega;\RR^3))\cap L_{p,\mu}(J_T;H_p^3(\Omega;\RR^3)),$$
then
$$
\varphi\in H_{p,\mu}^1(J_T;L_p(\Omega))\cap L_{p,\mu}(J_T;H_p^2(\Omega))=:\mathbb{E}_{1,\mu}^\varphi(J_T),
$$
since $d\in C(J_T;C^2(\overline{\Omega}))$ by our assumption \eqref{eq:Assum_for_p} on $p$ and $\mu$. Furthermore, we have $u\in C(J_T;C^1(\overline{\Omega}))$.
Hence, there exists a constant $C_1=C_1(T)>0$ such that for any $\tau\in (0,T]$ it holds that
\begin{align*}
\|G_1\varphi\|_{L_p(J_\tau\times\Omega)}&\le C_1\|\varphi\|_{L_p(J_\tau;H_p^1(\Omega))}
\le \tau^{1/p} C_1 \|\varphi\|_{L_\infty(J_\tau;H_p^1(\Omega))}
\le \tau^{1/p} C_1M_1\|\varphi\|_{\mathbb{E}_{1,\mu}^\varphi(J_\tau)},
\end{align*}
with some constant $M_1>0$, not depending  on $\tau$, since $\varphi(0)=0$ and
$$
\mathbb{E}_{1,\mu}^\varphi(J_\tau)\hookrightarrow C([0,\tau];W_p^{2\mu-2/p}(\Omega))\hookrightarrow C([0,\tau];H_p^{1}(\Omega)),
$$
by the assumption \eqref{eq:Assum_for_p} on $p$ and $\mu$. Furthermore, the trace space for $\partial_\nu\varphi$ on $\partial\Omega$ is
$$
\mathbb{G}_{\mu}(J_T):=W_{p,\mu}^{1/2-1/2p}(J_T;L_p(\partial\Omega))\cap L_{p,\mu}(J_T;W_{p}^{1-1/p}(\partial\Omega)).
$$
It follows from \cite[Lemma A.5 (ii)]{DuShaSi24} that there exists a constant $C_2=C_2(T)>0$,  such that for any $\tau\in (0,T]$, the estimate
\begin{align*}
  \|G_0\varphi\|_{\mathbb{G}_{\mu}(J_\tau)} \le C_2\|\beta\|_{\mathbb{G}_{\mu}(J_\tau)}\|\varphi\|_{\mathbb{F}_{\mu}^\varphi(J_\tau)}
  \le C_2M_2\|\beta\|_{\mathbb{G}_{\mu}(J_\tau)}\|\varphi\|_{\mathbb{E}_{1,\mu}^\varphi(J_\tau)},
\end{align*}
is valid, with some constant $M_2>0$, being independent of $\tau$, since $\varphi(0)=0$. Here we have used the fact that
$$
\mathbb{F}_\mu^\varphi(J_T)=\tr_{\partial\Omega}\mathbb{E}_{1,\mu}^\varphi(J_T)=W_{p,\mu}^{1-1/2p}(J_T;L_p(\partial\Omega))\cap L_{p,\mu}(J_T;W_{p}^{2-1/p}(\partial\Omega)),
$$
for any $T\in (0,\infty]$. Finally, let us note that $\|\beta\|_{\mathbb{G}_{\mu}(J_\tau)}\to 0$ as $\tau\to 0$, by the absolute continuity of the integral.

Since by assumptions (R) \& (P), $\gamma=\gamma(d,\nabla d)\in C(J_T\times \overline{\Omega})$ and $\gamma>0$, maximal $L_p$-regularity for the Neumann-Laplacian with inhomogeneous boundary conditions and a perturbation argument
imply that $\|\varphi\|_{\mathbb{E}_{1,\mu}^\varphi(J_\tau)}=0$ provided $\tau>0$ is sufficiently small. This in turn implies $\varphi(t)=0$ in its trace space
$$
X_{\gamma,\mu}^\varphi:=W_p^{2\mu-2/p}(\Omega)
$$
for all $t\in [0,\tau]$. We define
$$
\tau_*:=\sup\{\tau\in [0,T]\mid \forall\ t\in[0,\tau]:\varphi(t)=0\ \text{in}\ X_{\gamma,\mu}^\varphi\}.
$$
Suppose that $\tau_*<T$. Then we may solve \eqref{eq:Prbl_varphi} with initial time $\tau_*$ and initial value $\varphi(\tau_*)=0$. Applying again the maximal $L_p$-regularity for the
Neumann-Laplacian and the above perturbation argument, we obtain the existence of some $\tau>0$ such that $\varphi(t)=0$ in $X_{\gamma,\mu}^\varphi$ for all
$t\in [\tau_*,\tau_*+\tau]$, which is a contradiction to the maximality of $\tau_*$. This yields $\varphi(t)=0$ in $X_{\gamma,\mu}^\varphi$ for all $t\in [0,T]$.  Hence,
$\varphi(t,x)=0$ for all $(t,x)\in [0,T]\times\overline{\Omega}$, by the embedding $X_{\gamma,\mu}^\varphi\hookrightarrow C^1(\overline{\Omega})$ and \eqref{eq:Assum_for_p}. Note that $\varphi\equiv0$ if and only if $|d|_2\equiv 1$.

Finally, the fact that the solution $(u,d)$ of \eqref{eq:elgenincom} can be extended to a maximal interval of existence follows by  an iterated application of Proposition \ref{thm:LWP}.
\end{proof}

\begin{proof}[Proof of Theorem \ref{thm:contdep}]
For the proof, one may literally follow the lines of the proof of \cite[Proof of Theorem 5.1 (b)]{DuShaSi24}. We consider the equivalent system \eqref{eq:iEL3} with the definitions \eqref{eq:Def_sA}-\eqref{eq:Def_sF}. With a view on  \cite[Proof of Theorem 5.1 (b)]{DuShaSi24}, we need to know that
\begin{equation}\label{eq:contDep1}
[z\mapsto \mathcal{B}(z)=\sB(z)z]\in C^1(X_{\gamma,\mu},Y_{\gamma,\mu})
\end{equation}
and that for any $z_0\in \mathcal{M}$,
$$\mathcal{B}'(z_0)\in\mathcal{L}(X_{\gamma,\mu},Y_{\gamma,\mu}),$$
has a continuous right inverse, where
$$Y_{\gamma,\mu}:=\tr_{t=0}\mathbb{F}_{\mu}(J_T)=W_{p}^{2\mu-3/p}(\partial\Omega;\RR^3)$$
is the trace space of $\mathbb{F}_{\mu}(J_T)$ at $t=0$.

Adopting the strategy of the proof of \cite[Lemma B.3 (a)]{DuShaSi24}, \eqref{eq:contDep1} follows readily. The existence of a continuous right inverse for $\mathcal{B}'(z_0)$ can be
proven as in \cite[Proof of Proposition 2.5.1]{Mey10} in combination with Propositions \ref{prop:elliptic}, \ref{pro:LopShap_d} and \cite[Section 6.3.5 (iv)]{PrSi16}.
\end{proof}

\noindent
{\bf Acknowledgements}. Matthias Hieber acknowledges the support by the DFG Project FOR 5528.

\medskip

\noindent
\textbf{Data Availability Statement} Data sharing not applicable to this article as no
datasets were generated or analysed during the current study.

\medskip

\noindent
\textbf{Declarations}

\medskip

\noindent
\textbf{Conflict of interest} The authors have no conflict of interest to declare that are
relevant to the content of this article.


\end{document}